\documentclass[a4paper,11pt]{amsart}
\usepackage{amsfonts,amsmath,amssymb}
\usepackage{amssymb}
\usepackage{mathtools} 
\usepackage{verbatim}
\usepackage[dvips]{graphicx}
\usepackage{psfrag}
\usepackage{epsfig}
\usepackage{graphics}
\usepackage{subfig}
\usepackage{esint}
\theoremstyle{plain}
\newtheorem{theorem}{Theorem}[section]
\newtheorem{lemma}[theorem]{Lemma}
\newtheorem{proposition}[theorem]{Proposition}
\newtheorem{corollary}[theorem]{Corollary}
\newtheorem{remark}[theorem]{Remark}
\newtheorem{definition}[theorem]{Definition}
\newtheorem{example}[theorem]{Example}
\theoremstyle{definition}
\theoremstyle{remark}
\numberwithin{equation}{section}
\mathchardef\emptyset="001F
\newcommand{\U}{\mathcal{U}}

\newcommand{\e}{\varepsilon}
\newcommand{\Om}{\Omega}

\newcommand{\weakst}{\stackrel{\ast}{\rightharpoonup}}
\newcommand{\weak}{\rightharpoonup}

\newcommand{\R}{{\mathbb R}}
\newcommand{\F}{{\mathcal F}}
\newcommand{\G}{{\mathcal G}}

\renewcommand{\L}{{\mathcal L}}
\newcommand{\I}{{\mathcal I}}
\newcommand{\J}{{\mathcal J}}
\newcommand{\E}{{\mathcal E}}

\newcommand{\A}{{\mathcal A}}

\newcommand{\Z}{{\mathbb Z}}
\newcommand{\N}{{\mathbb N}}

\newcommand{\T}{{\mathcal T}}
\newcommand{\M}{{\mathbb M}}
\renewcommand{\H}{{\mathcal H}}
\renewcommand{\G}{{\mathcal G}}

\newcommand{\Mnn}{\M^{N\times N}}

\newcommand{\diag}{\hbox{{\rm diag}}}
\newcommand{\dsp}{\displaystyle}
\newcommand{\rank}{\mathrm{rank}}
\newcommand{\dist}{\mathrm{dist}}

\newcommand{\co}{{\rm co}}

\newcommand{\ut}{\tilde u}
\newcommand{\Ft}{\widetilde \F}
\renewcommand{\d}{\mathrm{d}}

\newcommand{\res}{\mathop{\hbox{\vrule height 7pt width .5pt depth 0pt \vrule height .5pt width 6pt depth 0pt}}\nolimits}
\newcommand{\ga}{\gamma}
\newcommand{\Ga}{\Gamma}
\newcommand{\LL}{\L}

\newcommand{\be}{\begin{equation}}
\newcommand{\ee}{\end{equation}}
\newcommand{\bes}{\begin{equation*}}
\newcommand{\ees}{\end{equation*}}
\newcommand{\integ}[3]{\int_{#1} #2 \, \d #3}

\renewcommand{\d}{\mathrm{d}}
\newcommand{\D}{\nabla} 

\newcommand{\eps}{\varepsilon}

\newcommand{\intern}[1]{\mathrm{int}\,#1}

\newcommand{\cell}{\textnormal{cell}}
\newcommand{\unolambda}{^{1,\lambda}}
\newcommand{\unouno}{^{1,1}}
\newcommand{\lambdalambda}{^{\lambda,\lambda}}
\begin{document}
\title[Interactions beyond nearest neighbours and rigidity of discrete energies]{Interactions beyond nearest neighbours and \\ rigidity of discrete energies: a compactness result \\ and an application to dimension reduction}
\author{Roberto Alicandro}
\author{Giuliano Lazzaroni}
\author{Mariapia Palombaro}
\address{Roberto Alicandro: DIEI, Universit\`a di Cassino e del Lazio meridionale, via Di Biasio 43, 03043 Cassino (FR), Italy}
\email{alicandr@unicas.it}
\address{Giuliano Lazzaroni: SISSA, Via Bonomea 265, 34136 Trieste, Italy}
\email{giuliano.lazzaroni@sissa.it}
\address{Mariapia Palombaro: University of Sussex, Department of Mathematics, Pevensey 2 Building, Falmer Campus,
Brighton BN1 9QH, United Kingdom}
\email{M.Palombaro@sussex.ac.uk}
%
%
\begin{abstract}
\small{
We analyse the rigidity of discrete energies where at least nearest and next-to-nearest neighbour interactions are taken into account.
Our purpose is to show that interactions beyond nearest neighbours have the role of penalising changes of orientation and, to some extent,
they may replace the positive-determinant constraint that is usually required when only nearest neighbours are accounted for.}
In a discrete to continuum setting, we prove a compactness result for a surface-scaled energy and we give bounds on its possible Gamma-limit.
\par
In the second part of the paper we follow the approach developed in the first part to study a discrete model for (possibly heterogeneous) nanowires. 
In the heterogeneous case, by applying the compactness result shown in the first part of the paper, we obtain
an estimate on the minimal energy spent to match different equilibria.
This gives insight into the nucleation of dislocations in epitaxially grown heterostructured nanowires.
\end{abstract}
\maketitle
{\small
\par
}
\bigskip
\section*{Introduction}
In atomistic models,  the behaviour of a system of  particles is usually described by pair interaction energies of the form
\[
\sum_{i\neq j} J(|u_i -u_j|) \,,
\]
where $i$ and $j$ label the pair of atoms, and $u_i$ and $u_j$ denote the corresponding positions. Typically, the interatomic potential $J$ is assumed to be 
repulsive at small distances and attractive at long distances, such as the celebrated Lennard-Jones potential. 
Numerical results (see for example \cite{YedBlaLeB} and references therein)
suggest that equilibrium configurations for such systems arrange approximately in a periodic lattice as the number of particles increases (crystallisation).
In the two-dimensional case, crystallisation has been proven in \cite{Th06}, where it is 
shown that ground states 
for a class of L-J type potentials  can be parametrised, up to rotations and translations, as the identity on a triangular lattice. 
Assuming that this reference parametrisation is maintained under deformations, effective energies can be derived in the limit as the atomic distance tends to zero and the number of particles tends to infinity (see e.g.\ \cite{ac,BlaLeBLio,braides-gelli}). 
It is expected that a meaningful macroscopic energy can be obtained by 
taking into account, for instance, only nearest-neighbour interactions in the reference lattice, 
while  the effect of  
long range interactions (i.e., between points that are distant in the reference lattice) can be somewhat neglected.
The restriction to nearest neighbours is usually complemented by technical assumptions that prevent 
the appearance of new ground states in the energy, typically, that the piecewise affine deformations defined by the value on the nodes of the triangulation satisfy a positive-determinant constraint. The latter condition (also called non-interpenetration) implies that deformations maintain the order of the periodic reference configuration and  in particular rules out undesirable foldings at the discrete level.    
However, at the same time it may produce ``unphysical'' effects that call into question 
the necessity of such an assumption; in this respect 
 an interesting discussion, in the context of fracture problems, can be found in the recent paper \cite{braides-gelli3}.
\par
The purpose of this paper is to analyse 
discrete systems when next-to-nearest neighbour interactions, and, more in general,  interactions in a finite range are taken into account, showing that to some extent their effect may replace the positive-determinant constraint in penalising changes of orientation, which are thus not excluded by assumption but rather energetically disfavoured.	
This question has been addressed in \cite{bs} in a one dimensional setting for the L-J potential. 
\par
We will restrict our attention to interaction potentials that satisfy polynomial growth conditions (i.e., strongly attractive at long distances), mimicking the behaviour of L-J type potentials at short distances  and leading in the macroscopic limit to continuum elastic theories that do not allow for fractures.  We believe that, even under this restriction, our analysis highlights interesting phenomena and provides an essential step towards the 
understanding of more general L-J type potentials.

As a model case, we consider energies of the form
\begin{eqnarray}\label{uno}
\sum_{|i- j|\leq R} \phi(|u_i -u_j|-|i-j|) \,,
\end{eqnarray}
where $R$ is a positive constant 
and   $\phi\colon\R\to [0,+\infty)$ is a potential satisfying polynomial growth conditions and such that $\min \phi(z)=\phi(0)=0$ and $\phi(z)>0$ if $z>0$. 
The constant $R$ determines the range of interactions contributing to the energy functional and is to be chosen according to the lattice under examination.
Assuming the existence of a reference configuration identified with a portion ${\mathcal L}'$ of a periodic lattice ${\mathcal L}$ in $\R^N$, the deformation map can be regarded as a function $u\colon {\mathcal L}'\to \R^N$.  A prototypical example is the two-dimensional case of particles sitting on a triangular lattice in their reference configuration and interacting via harmonic springs between nearest and next-to-nearest neighbours.  More precisely, normalising the equilibrium distance of the particles to one, the reference configuration ${\mathcal L}'$ is a portion of the hexagonal Bravais lattice $\Z v_1\oplus \Z v_2$, where $v_1=(1,0)$ and $v_2=(1/2,\sqrt{3}/2)$, and the corresponding total internal energy is of the form \eqref{uno} with 
$R=\sqrt{3}$ and $\phi(z) = z^2$.
Note that, up to translations, the ground states of energies of the type \eqref{uno} are given by all  linear maps in $O(N)$, while adding a positive-determinant constraint reduces them to $SO(N)$. Nonetheless, the presence of next-to-nearest neighbour or longer range interactions prevents the appearance of many changes of orientation, since these are energetically disfavoured. 
In contrast, it can be easily shown that the sole presence of nearest neighbour interactions allows 
changes of orientation without any additional cost. 

By scaling the reference lattice $\mathcal L$ by a small parameter $\e>0$ and identifying ${\mathcal L}'$ with $\e{\mathcal L}\cap\Om$, where $\Om$ is a bounded open set in $\R^N$,  one can consider a bulk scaling of \eqref{uno} and rewrite it in terms of difference quotients,  thus obtaining functionals of the form
\begin{eqnarray}\label{due}
F_\e(u):=\sum_{|i- j|\leq R}\e^N \phi\left(\frac{|u_i -u_j|}{\e }-|i-j|\right) \,,
\end{eqnarray}
where we use the notation $u_i:=u(\e i)$. The asymptotic behaviour of $F_\e$, as $\e$ tends to zero, was studied in \cite{ac} by means of $\Gamma$-convergence (see \cite{Braides,DM}) and leads to a continuum limit described by a functional  of the form $\int_\Om f(\nabla u)\, \d x$ defined on some Sobolev space. 
Here $f$ is a quasi-convex function, it is non-negative, and its minimum is equal to zero and is attained on $O(N)$. 
\par
Even though this result gives some insight into the structure of the equilibria of $F_\e$, encoded in the formula defining the density $f$, the effect of long range interactions in penalising changes of orientation takes place at a surface scale which is not detected by this analysis. Hence a higher order description is needed, which can be achieved by studying  the asymptotic behaviour of the surface-scaled energies 
\begin{equation}\label{riscalata}
E_\e(u):=\e^{-1}F_\e(u)\,.
\end{equation}
We prove a compactness result (Theorem \ref{thm:compactness}), asserting that the gradient of the limit of a sequence $u_\e$ for which $E_\e(u_\e)$  is uniformly bounded, is piecewise constant with 
values in $O(N)$ and that the underlying partition of $\Omega$ consists of sets of finite perimeter. Key mathematical tools in its proof are the well-known rigidity estimate of Friesecke, James, and M\"uller \cite{fjm} and the piecewise rigidity result proven in \cite{CGP}. 
\par
The characterisation of the $\Gamma$-limit of \eqref{riscalata} turns out to be a rather delicate problem. 
Propositions  \ref{propo:lower} and \ref{propo:up} provide bounds on the $\Gamma$-limit
in terms of interfacial energies that penalise changes of orientation. 
More precisely, denoting by $E$ the $\Gamma$-limit of a subsequence of $\{E_\e\}$, we show that, 
for each $u\in W^{1,\infty}(\Omega;\R^N)$ such that $\nabla u\in SBV(\Omega; O(N))$,
\begin{eqnarray}\label{lowbound}
E(u)\geq \int_{J_{\nabla u}}g_1(\nu_{\nabla u})\, \d{\mathcal H}^{N-1} \,.
\end{eqnarray}
Here  $J_{\nabla u}$ denotes the jump set of $\nabla u$,  $\nu_{\nabla u}$ is the unit normal to $J_{\nabla u}$, while $g_1$ is defined by a suitable asymptotic formula and is bounded from below by a positive constant (see Remark \ref{boundpositive}). An analogous upper bound holds for the class ${\mathcal W}(\Omega)$ of limiting deformations $u$ such that $J_{\nabla u}$ is a polyhedral set, that is, it consists of the intersection of $\Omega$ with the union of a finite number of $(N{-}1)$-dimensional simplices of $\R^N$. Namely, we have that
\begin{eqnarray}\label{upbound}
E(u)\leq \int_{J_{\nabla u}}g_2(\nu_{\nabla u})\, \d{\mathcal H}^{N-1} \quad \hbox {for all}\ u\in {\mathcal W}(\Omega) \,,
\end{eqnarray}
where $g_2$ is the limit of a sequence of suitable Dirichlet minimum problems and it is uniformly bounded from above by a positive constant. 
Thus, the continuum limit penalises the jump set $J_{\nabla u}$ and, at least on the class ${\mathcal W}(\Omega)$, is concentrated on $J_{\nabla u}$.
The computation of the $\Gamma$-limit of $E_\e$ remains an open question and leads to interesting analytical issues. Indeed, a standard argument to show that \eqref{lowbound} and \eqref{upbound} are optimal bounds and that  $g_1=g_2$  amounts to prove that it is possible to modify the boundary values of optimising sequences with a negligible energy cost. This does not seem a trivial task in the present context. Another interesting question, in analogy with density results in $BV$ spaces, is whether any admissible limiting deformation $u$ can be approximated by a sequence of regular deformations $u_n\in{\mathcal W}(\Omega)$, so that
$\int_{J_{\nabla u_n}}g_2(\nu_{\nabla u_n})\, \d{\mathcal H}^{N-1}$  converges to the corresponding energy of $u$. Indeed, by the lower semicontinuity of the $\Gamma$-limit, this would allow us first to extend the upper estimate \eqref{upbound} to the whole limiting domain, and second, in 
combination with a positive answer to the first question, to provide a complete characterisation of the 
$\Gamma$-limit. 
\par
Our discrete model is closely related to the
classical double-well  singuarly perturbed functionals  
studied
in the context of gradient theories for phase transitions (see e.g.\ \cite{Conti1,Conti-Schweizer}), where one considers energies of the form
\begin{equation}\label{double-well}
\int_\Omega \frac1\e W(\nabla u)+\e |\nabla^2 u|^2\, \d x \,.
\end{equation}
Here $W$ is a non-negative function vanishing on the set $K:=SO(N)A\cup SO(N)B$, where $A$ and $B$ are 
given rank-one connected matrices with positive determinant.  The second order term in \eqref{double-well} has the role of penalising oscillations between the two wells as  in our discrete model long range interactions penalise oscillations between $SO(N)$ and $O(N){\setminus} SO(N)$. 
In \cite{Conti-Schweizer} it is shown that the $\Gamma$-limit of \eqref{double-well} 
is an interfacial energy concentrated on the jump set of $\nabla u$.
A microscopic derivation of such result has been recently obtained 
in \cite{Angkana} in the context of square-to-rectangular martensitic phase transitions.
We point out that in \cite{Conti-Schweizer,Angkana} the two wells of $K$ 
consist of matrices with positive determinant, while this is not the case in the 
present context. Such difference is at the origin of the difficulties highlighted above.
\par
In the second part of the paper we follow the approach developed in the first part to study a discrete model for (possibly heterogeneous) nanowires. In this context we consider a different scaling of the energy, corresponding to a reduction 
of the system from $N$ dimensions to one dimension.  Specifically, we replace the factor $\e^N$ in
\eqref{due} by $\e$ and study the asymptotic behaviour of the corresponding energy defined by \eqref{riscalata}, namely
\begin{eqnarray}\label{nano-energy}
\E_\e(u):=\sum_{|i- j|\leq R} \phi\left(\frac{|u_i -u_j|}{\e }-|i-j|\right) \,.
\end{eqnarray}
The above sum is taken over a ``thin'' domain  (for the precise formula see \eqref{eng-eps});
as the lattice distance converges to zero, we perform a discrete to continuum limit and a dimension reduction simultaneously.
This model was first studied in \cite{LPS,LPS2} under the assumption that 
the admissible deformations satisfy the non-interpenetration condition. Here 
we remove such assumption and 
we  show that, by incorporating into the energy the effect of interactions in a certain finite range,
one can recover the  results of \cite{LPS,LPS2} 
and get even further insight into the problem. 
For the scaling of \eqref{nano-energy}, we obtain a complete description of the $\Gamma$-limit with 
respect to two different topologies (Theorems \ref{thm3} and \ref{thm4}).
It turns out that the $\Gamma$-limit with respect to the topology used in \cite{LPS,LPS2} is trivial 
(see Remark
\ref{rem100}),
that is, one can exhibit recovery sequences for which the gradient always lies in the same energy well up to an asymptotically vanishing correction. 
In order to see ``folding'' effects, we introduce a stronger topology which accounts for changes of orientations in the nanowire. 
In this case, we can prove that if we prescribe affine boundary conditions of the type $x\mapsto Bx$ with $\dist(B;SO(N))$ sufficiently small, 
then recovery sequences for minimisers will always preserve orientation (Remark \ref{ultimissimo}). 
In this respect our model is consistent with the non-interpenetration condition.
 On the other hand, we also show that minimisers may exhibit changes of orientation if  
we add to  the functionals loading terms of a certain form
(see Section \ref{sezione-forze} and Remark \ref{rem-forze}).
\par
The $\Ga$-limit is nontrivial, also in the weaker topology, when one considers heterogeneous nanowires, that consist of components with different equilibria,
arranged longitudinally; i.e., the interface between the components is a cross-section of the rod.
In this case,
we prove an estimate on the minimal energy spent to match the equilibria. Precisely, denoting by 
$k\in\N$ the number of atomic layers of the nanowire, we  show that the minimal cost grows faster 
than $k^{N-1}$. The proof of such result (Theorem \ref{thm:scaling}) follows as an application of the Compactness Theorem \ref{thm:compactness}
shown in the first part of the paper. 
Such lower bound is to be compared with the estimate that one can prove in the case of a 
two- or three-dimensional model 
accounting for dislocations. This is discussed in Section \ref{sec:disl} where we compare the 
minimal energy of heterogeneous defect-free systems and and the minimal energy of 
heterogeneous systems containing dislocations. It turns out that for sufficiently large values of 
$k$, the latter are energetically preferred 
since their energy may grow exactly like $k^{N-1}$ (see Remark \ref{rem:dislo}). 
In this respect our result is consistent with the one proven in \cite{LPS,LPS2}
under the non-interpenetration assumption.
We recall that the first variational justification of dislocation nucleation in 
nanowire heterostructures was obtained in \cite{mue-pal} in the context of non-linear elasticity. 
This result was later generalised to a discrete to continuum setting in \cite{LPS,LPS2} under the 
non-interpenetration condition, and is here validated without the latter assumption.  
More recently, variational models for misfit dislocations at semi-coherent interfaces  and in elastic thin films have been proposed in \cite{fpp} and \cite{fflm} respectively. 
\par
The paper is organised as follows.
In the first section, we define our interaction energy employing the notion of 
Kuhn decomposition of a cubic cell and we prove some technical lemmas on the rigidity of such energy. We observe that our setting includes specific lattices in dimension two and three in Remarks \ref{rmk:lattices1} and \ref{rmk:lattices}.
After proving our Compactness Theorem \ref{thm:compactness}, we provide some bounds on the possible $\Ga$-limits (Propositions \ref{propo:lower} and \ref{propo:up}).
The second section of the paper is devoted to nanowires; the results are stated in the general case of heterogeneous nanowires.
We introduce the minimal costs to bridge the equilibria and study their dependence on the thickness of the nanowire.
Afterwards, performing a discrete to continuum limit and a dimension reduction simultaneously,
we characterise the $\Ga$-limit of the energy functional for different choices of the topology 
(Theorems \ref{thm3} and \ref{thm4}).
We also discuss the effect of boundary conditions on the $\Gamma$-limit 
and briefly study a model including external forces (only in the homogeneous case, for simplicity).
In the final part of the paper, we compare the model for defect-free nanowires with models including dislocations at the interface,
showing that the latter are energetically favoured.
\subsection*{Notation}
We recall some basic notions of geometric measure theory for which we refer to \cite{AFP}. 
Given a bounded open set $\Om\subset\R^N$, $N\ge 2$, and $M\ge1$, $BV(\Om;\R^M)$ denotes the space of functions of bounded variation;
i.e., of functions $u\in L^{1}(\Om;\R^M)$ whose distributional gradient
$\mathrm{D} u$ is a Radon measure on $\Om$ with $|\mathrm{D} u|(\Om)<+\infty$,
where $|\mathrm{D} u|$ is the total variation of $\mathrm{D} u$.
If $u\in BV(\Om;\R^M)$,
the symbol $\nabla u$ stands for the density of the absolutely continuous
part of $\mathrm{D} u$ with respect to the $N$-dimensional Lebesgue measure $\L^N$.
We denote by  $J_u$ the jump set of $u$, by $u^+$ and $u^-$ the traces of $u$ on $J_u$,
and by $\nu_u(x)$ the measure theoretic inner normal to $J_u$ at $x$, which is defined for $\H^{N{-} 1}$-a.e.\ $x \in J_u$, where $\H^{N{-}1}$
is the $(N{-}1)$-dimensional Hausdorff measure. 
A function $u\in BV(\Om;\R^M)$ is said to be a special function of bounded variation if $\mathrm{D} u - \nabla u \, \L^N$ is concentrated on $J_u$; in this case one writes $u\in SBV(\Om;\R^M)$.
Given a set $E\subset \Om$,  we denote by $P(E,\Om)$ its  relative perimeter in $\Omega$ and by $\partial^*E$ its reduced boundary.
\par
For $N\ge2$, $\Mnn$ is the set of real $N{\times} N$ matrices,
$GL^+(N)$ is the set of matrices with positive determinant,
$O(N)$ is the set of orthogonal matrices, and $SO(N)$ is the set of rotations.
We denote by $I$ the identity matrix and $J$ the reflection matrix such that $Je_1=-e_1$ and $Je_i=e_i$ for $i=2,\dots,N$, where $\{e_i\colon i=2,\dots,N\}$ is the canonical basis in $\R^N$.
The symbol $\co(X)$ stands for the convex hull of a set $X$ in $\Mnn$.
Moreover, given $N{+}1$ points $x_0, x_1,\dots, x_N\in\R^N$, we denote by $[x_0, x_1,\dots, x_N]$ the simplex determined 
by all convex combinations of those points.
\par
Finally, $\U$ is the class of subsets of $(-L,L)$ 
that are disjoint union of a finite number of open intervals.
\par
In the paper, the same letter $C$ denotes various positive constants whose precise value may change from place to place. 
\par
%
%
%
\section{Surface scaling regime for discrete energies}\label{sec:comp}
We study the deformations of a Bravais lattice governed by pairwise potentials 
with finite range interactions. Up to an affine deformation $H\in GL^+(N)$, we can reduce to the case where the lattice is $\Z^N$. 
(See Remark \ref{rmk:lattices} for details on the treatment of some specific lattices in dimension two and three.) 
In order to define the interaction energy, we   
introduce the so-called Kuhn decomposition, denoted by $\T$, which consists in
a partition of $\R^N$ into $N$-simplices.
\begin{figure} 
\centering
\includegraphics[width=.7\textwidth]{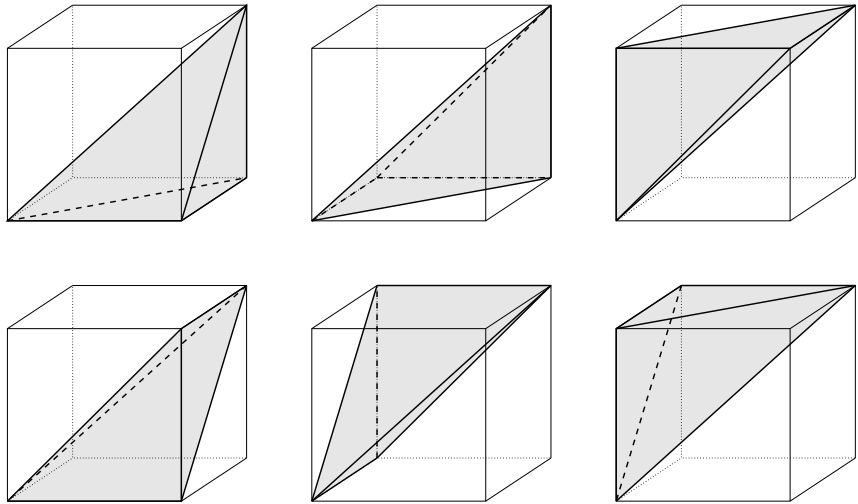}
\caption{
The six tetrahedral elements of the Kuhn decomposition of a cube in dimension three.
}
\label{fig:kuhn}
\end{figure}
First we define a partition $\T_0$ of the unit cube $(0,1)^N$ into $N$-simplices: 
we say that 
$T\in\T_0$ if the $(N{+}1)$-tuple of its vertices belongs to the set
\begin{equation*}
\left\{  
\{0, e_{i_1},e_{i_1}+ e_{i_2}, \dots,e_{i_1}+ e_{i_2}+ \dots + e_{i_N} \} \colon
\begin{pmatrix}
1 & 2 & \cdots & N \\
 i_1 & i_2 & \cdots & i_N
\end{pmatrix}\in S_N
\right\},
\end{equation*}
where $S_N$ is the set of permutations of $N$ elements; see Figure \ref{fig:kuhn}.
Next, we define $\T$ as the periodic extension of $\T_0$ to all of $\R^N$. 
We say that two nodes $x,y\in \Z^N$ are contiguous if there exists a simplex $T\in\T$ 
that has both $x$ and $y$ as its vertices.
We set 
\begin{equation}\label{bonds1}
B_1:=\{
\xi\in\R^N \colon x \text{ and }x+\xi \text{ are contiguous}
\}\,.
\end{equation}
If both $[x_0, x_1,\dots, x_N]$ and $[y_0, x_1,\dots, x_N]$ belong to $\T$, then 
we say that $[x_0, x_1,\dots, x_N]$ and $[y_0, x_1,\dots, x_N]$ are neighbouring simplices 
(i.e., they share a facet) and
$x_0$ and $y_0$ are opposite vertices.
We set 
\begin{equation}\label{bonds2}
B_2:=\{
\xi\in\R^N \colon x \text{ and }x+\xi \text{ are opposite vertices}
\}\,,
\end{equation}
and remark that, by periodicity, $B_1$ and $B_2$ do not depend on $x$.
\par
Let $\Om$ be an open Lipschitz subset of $\R^N$.
Given $\e>0$ we consider 
\bes
\LL_\e := \e\Z^N \cap \overline\Om_\e\,, 
\ees
where $\overline\Om_{\e}$ is the union of all hypercubes with vertices in $\e\Z^N$ that have 
non-empty intersection with $\Om$.
We identify every deformation $u$ of the lattice $\LL_\eps$ 
by its piecewise affine interpolation with respect to the triangulation $\e\T$. 
By a slight abuse of notation, such extension is still denoted by $u$. 
We define the domain of the functional as
\bes
\begin{split}
\A_{\eps}:= \big\{ u\in C^0(\overline\Om_{\eps};\R^N) \colon & u \ 
\text{piecewise affine,}\\ 
& \D u \  \text{constant on}\ \overline\Om_{\eps}\cap \e T\  \forall\, T\in\T \big\} \,.
\end{split}
\ees
For fixed $p>1$ and $H\in GL^+(N)$, we study the following surface-scaled discrete energy,
\begin{equation}\label{surf-eng}
E_\e(u) := \e^{N-1}\sum_{\xi\in B_1\cup B_2} \!\!\!\sum_{\substack{x\in\LL_\e\\x{+}\e\xi\in\LL_\e}}\!\!\! \Big| \frac{|u(x{+}\e\xi){-}u(x)|}{\e}-|H \xi| \Big|^p
\,, 
\end{equation}
for $\e>0$.
%
%
\par
%
%
%
We underline that our results generalise to energies of the form 
\be\label{eng-general}
\e^{N-1}\sum_{\substack{\xi\in\Z^N\\|\xi|\le R}} \, \sum_{\substack{x\in\LL_\e\\x{+}\e\xi\in\LL_\e}}\!\!\! \phi \Big( \xi, \frac{|u(x{+}\e\xi){-}u(x)|}{\e}-|H \xi| \Big) \,, 
\ee
where $R$ is chosen in such a way that $R \ge \max \{ |\xi| \colon \xi\in B_1\cup B_2 \}$,
$\phi\colon\Z^N\times\R\to [0,+\infty)$ is a positive potential satisfying polynomial growth conditions in the second variable and such that $\min \phi(\xi,z)=\phi(\xi,0)=0$ and $\phi(\xi,z)>0$ if $z>0$. 
For simplicity, here we develop our analysis in detail only in the case of $p$-harmonic potentials as in \eqref{surf-eng}.
%
%
%
%
\subsection{Discrete rigidity}
The following result will play a crucial role in deriving rigidity estimates in our discrete setting.
(See \cite{BSV,Schm06,Th06} for discrete rigidity estimates.)
\begin{theorem}\label{thm-rigidity}
\cite[Theorem 3.1]{fjm}
Let $N\geq 2$, and let $1< p < \infty$.
Suppose that $U\subset\R^{N}$ is a bounded Lipschitz domain. 
Then there exists a constant $C=C(U)$
such that for each $u\in W^{1,p}(U;\R^{N})$ 
there exists a constant matrix $R\in SO(N)$ such that 
\begin{equation}\label{rigidity}
\|\D  u-R\|_{L^{p}(U;\Mnn)} \leq C(U)
\|\dist(\D  u,SO(N))\|_{L^{p}(U)} \,.
\end{equation}
The constant $C(U)$ is invariant under dilation and translation of the domain. 
\end{theorem}
It is convenient to define the energy of a single simplex $T$ with vertices $x_0,\dots,x_N$,
\[
E_\cell(u_F;T):= \sum_{i\le j=0}^N \Big| |F(x_i-x_j)|-|H(x_i-x_j)| \Big|^p \quad \text{for every } F\in \Mnn \,,
\]
where $u_F$ is the affine map $u_F(x):=Fx$.
The following lemma provides a lower bound on $E_\cell(u_F;T)$ in terms of the distance of $F$ from $O(N)$. 
It will be instrumental in using Theorem \ref{thm-rigidity}.
%
%
%
\begin{lemma}\label{lemma-equiv}
There exists a constant $C>0$ such that 
\begin{subequations}\label{bound-from-below}
\begin{align}
\dist^p(F,SO(N)H) & \leq  C \, E_\cell(u_F;T)  \quad \text{for every} \ F\in \Mnn \ \text{with}\ \det F \ge0 \,, \\
\dist^p(F,(O(N){\setminus}SO(N))H) & \leq  C \, E_\cell(u_F;T)  \quad \text{for every} \ F\in \Mnn \ \text{with}\ \det F \le0 \,. 
\end{align}
\end{subequations}
%
\end{lemma}
%
%
\begin{proof}
Set $\delta_{ij}:= |F(x_i-x_j)|-|H(x_i-x_j)|$ for $i\le j=0,\dots, N$, so that
$E_\cell(u_F;T)  = \sum_{i\le j=0}^N \delta_{ij}^p$.
Assume first that $\dist(F,O(N)H)$ is small and that $\det F\ge0$.
In particular, we can assume $\dist(F,SO(N)H)\le \tau$ where $\tau$ is a small parameter whose value will be fixed later.
By the equivalence of norms in $\R^{(N+1)(N+2)/2}$, it suffices to prove 
\[
\frac1C \dist^2(F,SO(N)H) \leq  \sum_{i\le j=0}^N \delta_{ij}^2 =: E_2(F) \,. 
\]
Define $RH$ as the orthogonal projection of $F$ on $SO(N)H$, so that $\dist(F,SO(N)H)=|F-RH|$.
By computing the second order Taylor expansion of $E_2$ about $RH$ and recalling that matrices in $SO(N)H$ are minimum points for $E_2$, we see that
\[
 E_2(F) = \frac{1}{2}\nabla^2 E_2(RH)(F{-}RH,F{-}RH) + o(|F-RH|^2)  \ge C |F-RH|^2 + o(|F-RH|^2) \,,
\]
since on $SO(N)H$ the Hessian of $E_2$ is positive definite on the orthogonal complement of the tangent space of 
$SO(N)H$ at $RH$,
see e.g.\ \cite[Remark 4]{BSV}.
In the case when $\det F\le0$ the above argument is repeated replacing 
$SO(N)$ by $O(N){\setminus} SO(N)$.
Therefore, if $\dist(F,O(N)H)$ is sufficiently small, 
then \eqref{bound-from-below} is readily seen to hold.
\par
On the other hand, if $\dist(F,O(N)H)$ is sufficiently large, (and therefore $\max_{i\le j}|\delta_{ij}|$ is larger than a fixed constant,) then
\begin{align*}
\dist^2(F,SO(N)H ) & \leq |F-H|^2 \leq C \sum_{j=1}^N |(F-H)(x_j-x_0)|^2 \,, \\
\dist^2(F,(O(N){\setminus}SO(N))H ) & \leq |F-JH|^2 \leq C \sum_{j=1}^N |(F-JH)(x_j-x_0)|^2 \,.
\end{align*}
By the triangle inequality, $|(F-H)(x_j-x_0)|\leq |F(x_j-x_0)|+|H(x_j-x_0)| = 2|H(x_j-x_0)|+
|\delta_{0j}| \le C \max_{i\le j}|\delta_{ij}|$,  
and the same holds for $F-JH$;
thus \eqref{bound-from-below} follows. The intermediate cases follow by a continuity argument.
%
\end{proof}
The next lemma asserts that if in two neighbouring simplices the sign of $\det \nabla u$ has different sign, 
then the energy of those two simplices is larger than a positive constant.
It will be convenient to define the energetic contribution of the interactions within two neighbouring 
simplices $T=[x_0,x_1,\dots,x_N]$, $S= [y_0,x_1,\dots,x_N]$ as 
\begin{equation*}
\begin{split}
E_\cell(u;S\cup T) := & \sum_{i\le j=0}^N \Big| | u(x_i)-u(x_j)|-|H(x_i-x_j)| \Big|^p +
\sum_{j=1}^N \Big| | u(y_0)-u(x_j)|-|H(y_0-x_j) | \Big|^p  \\
& +
\Big| | u(y_0)-u(x_0)| - |H(y_0-x_0)| \Big|^p.
\end{split}
\end{equation*}
\begin{lemma}\label{costo-inversione}
There exists a positive constant $C_0$ (depending on $H$) with the following property:
if two neighbouring $N$-simplices $S$, $T$ have different orientations in the deformed configuration, i.e.,
\[
\det \left( \nabla u |_{{S}} \right)
\det \left( \nabla u |_{{T}} \right)
\le0 \,,
\]
then $E_\cell(u;S\cup T)\ge C_0$.
\end{lemma}
\begin{figure} 
\centering
\subfloat[]{
\psfrag{a1}{$x_0$}
\psfrag{a2}{$x_1$}
\psfrag{a3}{$x_2$}
\psfrag{a4}{$x_3$}
\psfrag{a5}{$x_4$}
\psfrag{a6}{$x_5$}
\includegraphics[width=.35\textwidth]{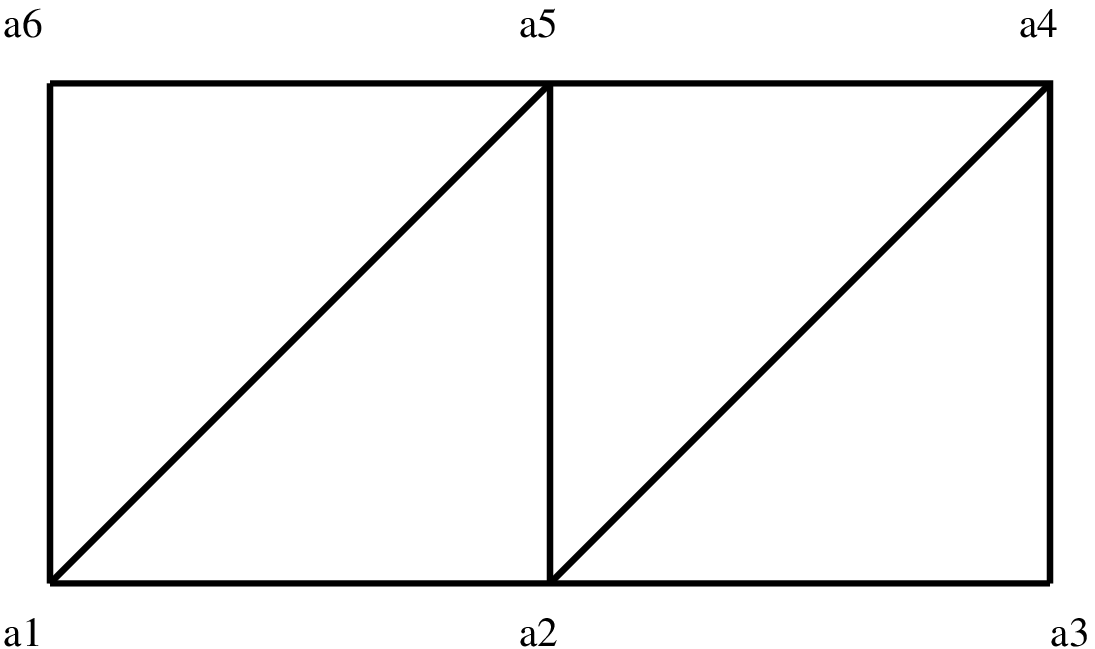}
}
\hspace{.1\textwidth}
\subfloat[]{
\psfrag{a1}{$x_0$}
\psfrag{a2}{$x_1$}
\psfrag{a3}{$x_2$}
\psfrag{a4}{$x_3$}
\psfrag{a5}{$x_4$}
\psfrag{a6}{$x_5$}
\psfrag{a7}{$x_6$}
\psfrag{a8}{$x_7$}
\psfrag{a9}{$x_8$}
\psfrag{a10}{$x_{9}$}
\psfrag{a11}{$x_{10}$}
\psfrag{a12}{$x_{11}$}
\includegraphics[width=.28\textwidth]{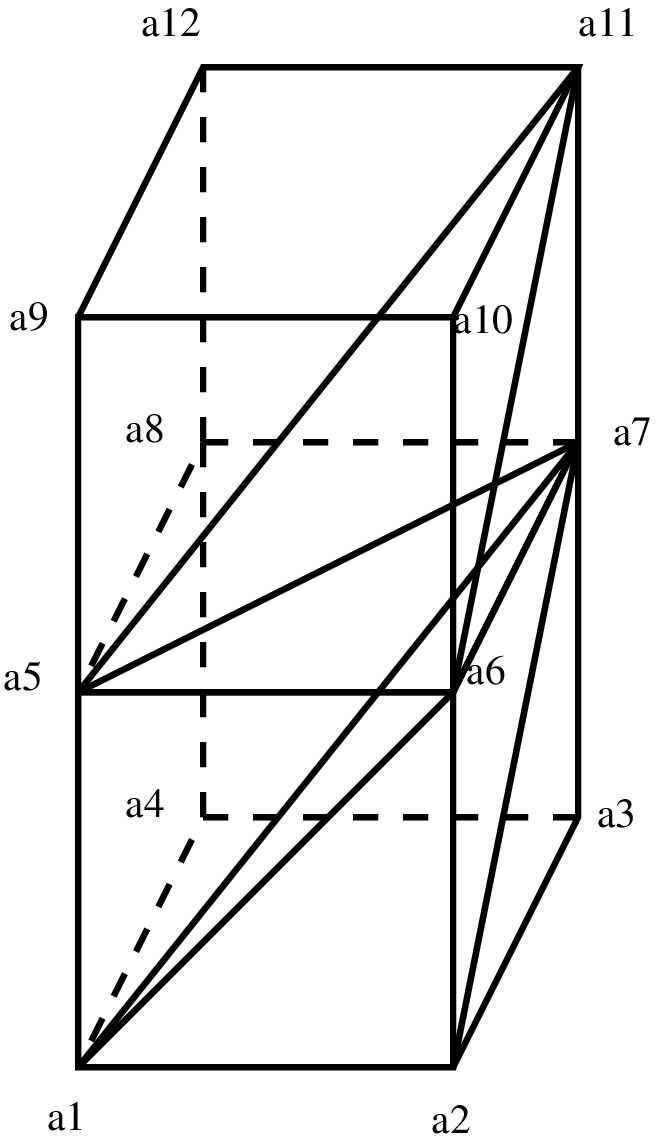}
}
\caption{We show a detailed proof of Lemma \ref{costo-inversione} for $H=I$ in dimension two and three. \newline
(\textsc{a}) 
Let $T=[x_0,x_1,x_4]$ and $S=[x_0,x_4,x_5]$.
In the case when $\big|\nabla u|_T - I \big| <C\tau $ and 
$\big|\nabla u|_S - Q \big| <C\tau $, we find that $| u(x_5) - u(x_1) | < C \tau $,
which implies
$E_\cell(u;S\cup T)\ge 1$ for $\tau$ is sufficiently small.
 In the case when $T=[x_0,x_1,x_4]$, $S=[x_1,x_3,x_4]$,
$\big|\nabla u|_T - I \big| <C\tau $, and $\big|\nabla u|_S - Q \big| <C\tau $, 
we have that $| u(x_3) - x_5 | <C\tau $ 
(assuming  w.l.o.g.\ that $u(x_4)=x_4$). Then 
$| u(x_3) - u(x_0) | < \sqrt{5} - 1$ and therefore  $E_\cell(u;S\cup T)\ge 1$ for small $\tau$. \newline
(\textsc{b})  
In the case when $T=[x_0,x_1,x_5,x_6]$, $S=[x_0,x_4,x_5,x_6]$,
$\big|\nabla u|_T - I \big| <C\tau $, and $\big|\nabla u|_S - Q \big| <C\tau $, 
we have that $| u(x_4) - u(x_1) | < C \tau $, 
which yields $E_\cell(u;P\cup Q)\ge 1$.
In the case when $T=[x_0,x_4,x_5,x_6]$, $S=[x_4,x_5,x_6,x_{10}]$,
$\big|\nabla u|_T - I \big| <C\tau $, and $\big|\nabla u|_S - Q \big| <C\tau $, 
we have that $| u(x_{10}) - x_2 | < C\tau $ 
 (assuming w.l.o.g.\ that  $u(x_{6}) = x_6$). 
 Then 
$| u(x_0) - u(x_{10}) | < \sqrt{6} -1$, which gives     
$E_\cell(u;P\cup Q)\ge 1$ for $\tau$ small.
}
\label{costo-inv2}
\end{figure}
\begin{proof}
We first consider the case when $H=I$ is the identity matrix.
Let $\tau$ be a small positive constant. 
If  $\max \{\dist\left( \nabla u |_{{S}}, O(N) \right), \dist\left( \nabla u |_{{ T}}, O(N) \right)\}>\tau$, then, by Lemma \ref{lemma-equiv},  $E_\cell(u;S\cup T)\ge C\tau^p$. Otherwise
we can assume that $\big|\nabla u|_{T} - I \big|< C\tau $ and  $\big|\nabla u|_{S} - Q \big|< C\tau $, 
where $Q\in O(N){\setminus} SO(N)$ is the reflection across the facet shared by $S$ and $T$.
If $S$ and $T$ are neighbouring simplices within the same unit cube, then $|x_0 - y_0| = \sqrt{2}$, since  $x_0$ and $y_0$ are opposite vertices of a two-dimensional facet of the unit cube, 
while $|u(x_0) - u(y_0)| < C \tau$. This yields  $E_\cell(u;S\cup T)\ge 1$ for 
$\tau$ sufficiently small.
If $S$ and $T$ belong to different cubes, then $|x_0 - y_0| = \sqrt{N-1 + 4}$, 
while $|u(x_0) - u(y_0)| < \sqrt{N-1} + C \tau$, which gives $E_\cell(u;S\cup T)\ge 1$ for small $\tau$. (See Figure \ref{costo-inv2} for the proof in dimension two and three.)
The case of a general $H$ is recovered by applying the previous argument to $v(x):=u(H^{-1}x)$.
\end{proof}
We conclude the discussion about discrete rigidity with some remarks on the choice of the interactions.
\begin{figure} 
\centering
\subfloat[]{
\includegraphics[width=.23\textwidth]{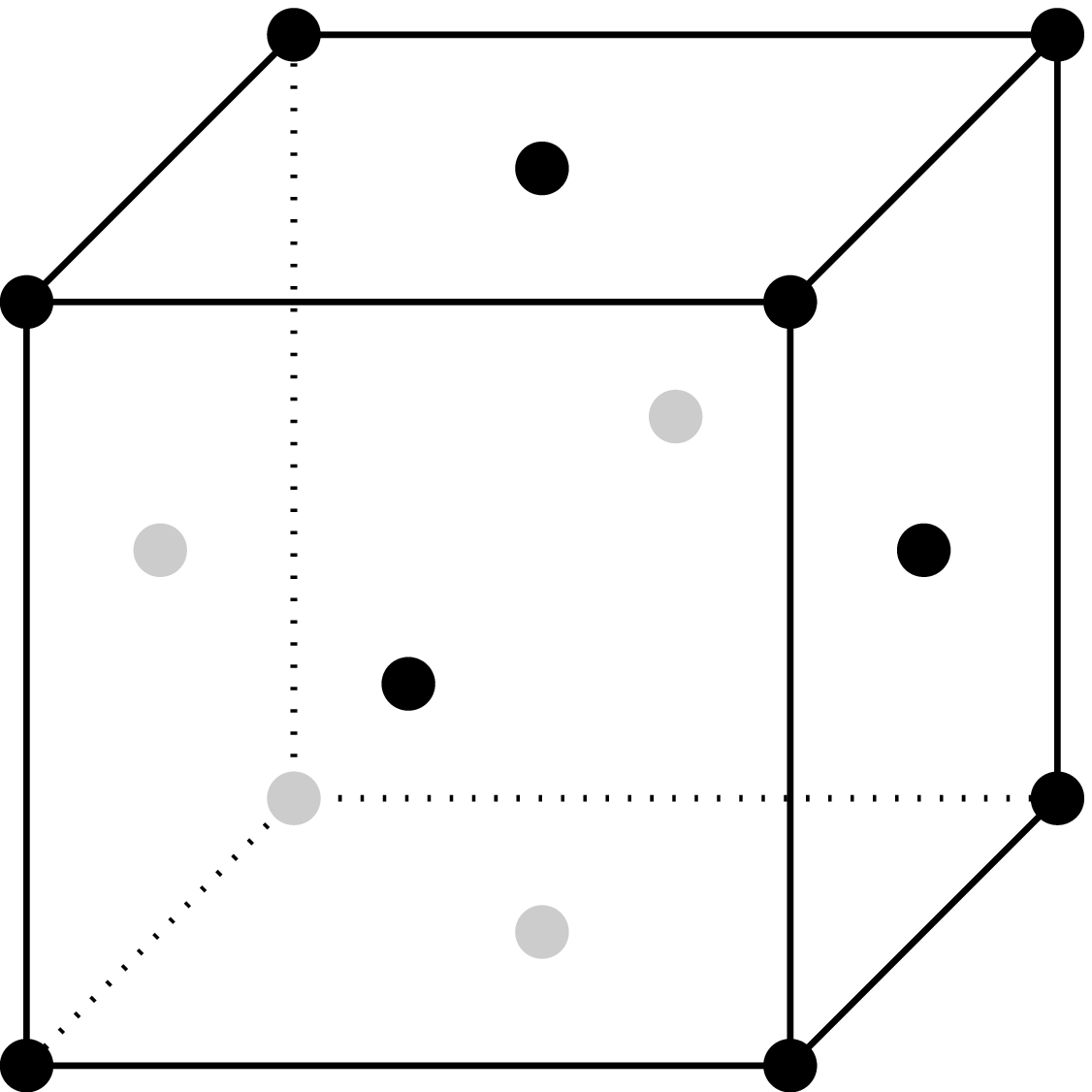}
}
\hspace{.1\textwidth}
\subfloat[]{
\includegraphics[width=.23\textwidth]{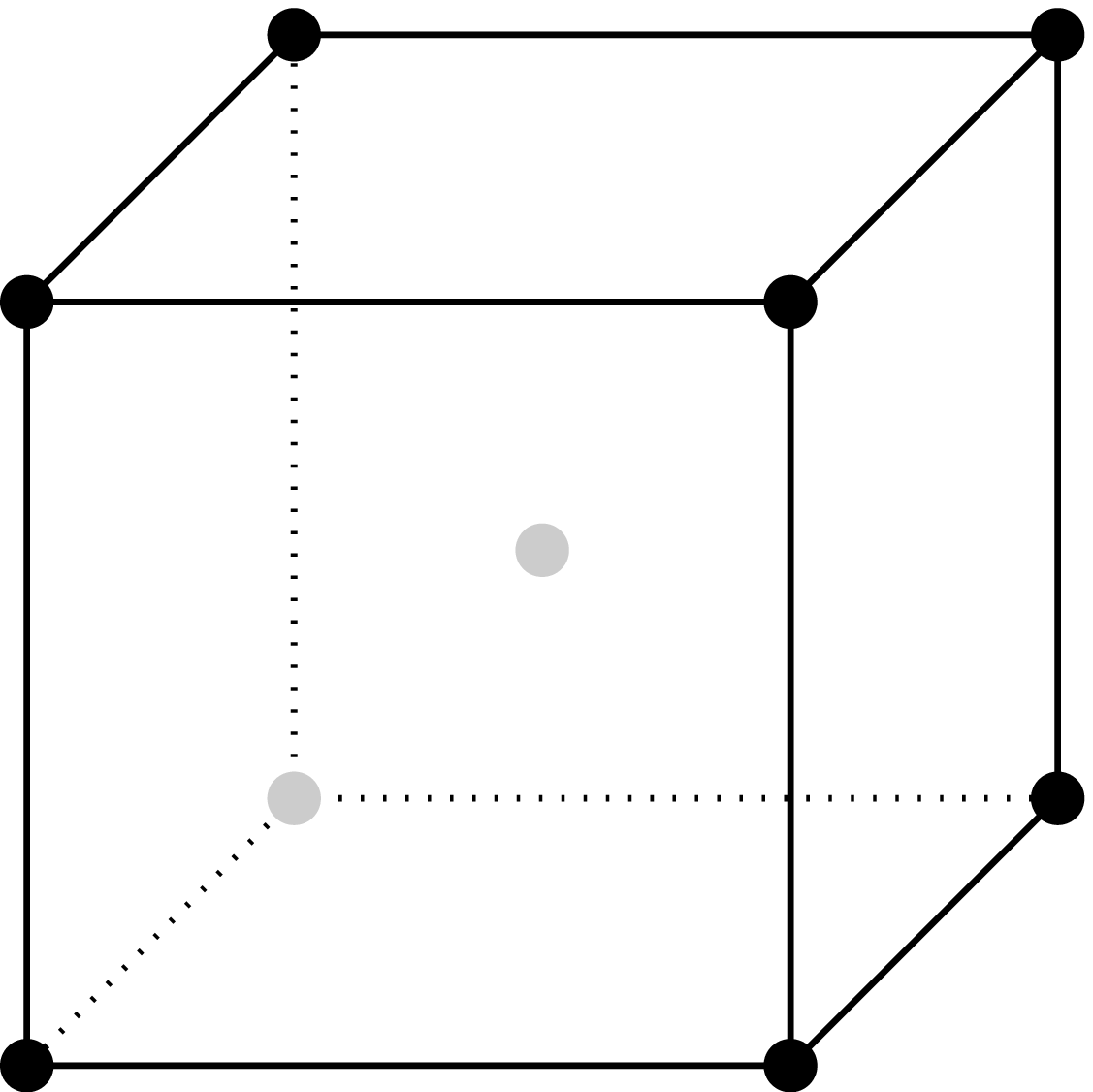}
}
\caption{
Cubic cells of the face-centred cubic (\textsc{a}) and of the body-centred cubic lattice (\textsc{b}).
}
\label{fcc-bcc}
\end{figure}
\begin{remark}\label{rmk:lattices1}
 For $N=2$ and $H=I$, using the Kuhn decomposition we model a square lattice with bonds given by the sides and the diagonals of each cell. Notice that one of the diagonals is accounted for in $B_1$, while the other in $B_2$; the other bonds in $B_2$ correspond to longer range interactions. Further interactions could be added to the total energy in order to make the bonds symmetric. More precisely, one could define the total interaction energy as
\[
 E(u;\Om) := \sum_{\xi\in B(M) }\!\! \sum_{\substack{x\in\overline\Om\cap\Z^N\\x+\xi\in\overline\Om\cap\Z^N}}\!\! \Big| |u(x{+}\xi){-}u(x)|-|\xi| \Big|^p \,,
\]
where $B(M):=\{\xi\in\Z^N: |\xi|\le M \}$; for $M=\sqrt3$, $B(\sqrt5)\supset B_1\cup B_2$. We underline that one retrieves the same rigidity properties stated above also by choosing $M=2$; i.e., the Kuhn decomposition is not ``optimal'' in this case. In general, the choice of $M$ depends on $N$ and $H$.
\end{remark}
The Kuhn decomposition is relevant especially for modeling some specific Bravais lattices as observed in the following remark.
\begin{remark} \label{rmk:lattices}
We show how the Kuhn decomposition can be used to parametrise Bravais lattices that are related with the crystalline structure of metals. We recall that a Bravais lattice in $\R^N$ consists of all integer combinations of $N$ linearly independent vectors, called generators.
\par
For $N=2$, the Bravais lattice generated by $v_1=(1,0)$ and $v_2=(\frac12, \frac{\sqrt3}2)$ is called hexagonal (or equilateral triangular) since every point has six nearest neighbours at distance one; moreover, every point has six next-to-nearest neighbours at distance $\sqrt3$.
In order to map $\Z^N$ onto the hexagonal lattice, we set
\[
H=\begin{pmatrix*}[r] 
1 & -\tfrac12 \\
0 & \tfrac{\sqrt3}2
\end{pmatrix*} \,,
\]
so that $He_1=v_1$ and $H(e_1+e_2)=v_2$.
One can see that $H$ establishes a bijection between vectors in $B_1$, respectively $B_2$, and vectors associated with nearest neighbour interactions, respectively next-to-nearest, in the hexagonal lattice; cf.\ \eqref{bonds1}--\eqref{bonds2} for the definition of $B_1$ and $B_2$. 
%
\par
In dimension three, a structure of interest is the face-centred cubic lattice, which is the Bravais lattice generated by $v_1=(0,\frac12,\frac12)$, $v_2=(\frac12,0,\frac12)$, and $v_3=(\frac12,\frac12,0)$; see Figure \ref{fcc-bcc}(\textsc{a}). Such lattice determines a subdivision of the space into cubic cells of edge one, where the atoms occupy the vertices and the centres of the facets of each cell. Each point has twelve nearest neighbours at distance $\frac{\sqrt2}2$ and six next-to-nearest neighbours at distance one. Nearest and next-to-nearest neighbour interactions guarantee the rigidity of the lattice; i.e., a deformation preserving the length of nearest and next-to-nearest bonds needs to be a rotation of the original lattice.
Setting $He_1=v_1$, $H(e_1+e_2)=v_2$, and $H(e_1+e_2+e_3)=v_3$, we obtain
\[
H= \tfrac12 \begin{pmatrix*}[r] 
0 & 1 & 0 \\
1 & -1 & 1 \\
1 & 0  & -1 
\end{pmatrix*} \,.
\]
Under this affinity, the bonds in $B_1$ associated with the Kuhn decomposition are transformed into twelve nearest and two next-to-nearest neighbour interactions for the face-centred cubic lattice; the images of the bonds in $B_2$ include four more next-to-nearest neighbour interactions. The total energy defined via the Kuhn decomposition includes few more interactions of longer range.
\par
We conclude with the body-centred cubic lattice, which is generated by $v_1=(-\tfrac12,-\tfrac12,-\tfrac12)$, $v_2=(0,0,1)$, $v_3=(0,1,0)$; see Figure \ref{fcc-bcc}(\textsc{b}). 
Here the atoms occupy the vertices and the centre of cubic cells of edge one. Arguing as above we get
\[
H= \tfrac12 \begin{pmatrix*}[r] 
-1 & 1 & 1 \\
1 & -1 & 1 \\
1 & 1  & -1 
\end{pmatrix*} \,.
\]
Applying the transformation $H$, the fourteen bonds in $B_1$ are mapped into eight vectors of length $\frac{\sqrt{3}}{2}$ and six of length one; these correspond exactly to the nearest neighbour interactions in the body-centred cubic lattice, if the definition of the neighbours is based on a Delaunay triangulation, see \cite{LPS2} for details. The twelve bonds in $B_2$ are in bijection with vectors corresponding to the next-to-nearest neighbour interactions for that triangulation.
\end{remark}
%
\subsection{Compactness result}
Before stating our main result, we recall that a partition $\{E_i\}_{i\in\N}$ of $\Om$ is said a Caccioppoli partition if
$\sum_{i\in\N} P(E_i,\Om)<+\infty$, where $P(E_i,\Om)$ denotes the perimeter of $E_i$ in $\Om$. Given a rectifiable set $K\subset\Om$, we say that a Caccioppoli partition $\{E_i\}_{i\in\N}$ of $\Om$ is subordinated to $K$ if for every $i\in\N$
the reduced boundary $\partial^* E_i$
of $E_i$ is contained in $K$, up to a
$\H^{N-1}$-negligible set.
\begin{theorem}[Compactness] \label{thm:compactness}
Let $\{u_\e\}\subset\A_\eps$ be a sequence such that 
\be\label{equibounded}
E_\e(u_\e) < C \,.
\ee 
Then there exist a subsequence (not relabelled) and a function $u\in W^{1,\infty}(\Om;\R^N)$ such that $\nabla u_\e \to \nabla u$ in $L^p(\Om;\Mnn)$ and
\be\label{grad-SBV}
\nabla u\in SBV(\Om; O(N)H) \,.
\ee
%
Specifically, 
$u$ is a collection of an at most countable family of rigid 
deformations, i.e., there exists a Caccioppoli partition $\{E_i\}_{i\in\N}$ subordinated to 
the reduced boundary
$\partial^* \{\nabla u \in SO(N)H\}$, such that  
\be \label{pw-rigidity}
u(x) = \sum_{i\in\N} (R_i H x + b_i)\chi_{E_i}(x)\,,
\ee
where $R_i\in O(N)$ and $b_i\in\R^N$. 
Moreover, if $\partial^*E_i\cap\partial^*E_j\neq\emptyset$, then $\det R_i \, \det R_j =-1$ and $\partial^*E_i\cap\partial^*E_j$ is flat.
\end{theorem}
\begin{proof}
Note that by Lemma \ref{lemma-equiv}
\begin{eqnarray}\label{stimadist}
\int_\Om\dist^p(\nabla u_\e, O(N)H)\, \d x\leq C\e\,.
\end{eqnarray}
In particular, $\|\nabla u_\e\|_{L^p(\Om)}<C$
and therefore, up to subsequences, 
$ u_\e -\dsp\fint_\Om u_\e\, \d x\weak u$ weakly in  $W^{1,p}(\Om;\R^N)$, for some $u\in W^{1,p}(\Om;\R^N)$. We first prove that $\nabla u\in O(N)H$ (which implies that $u\in W^{1,\infty}(\Omega;\R^N)$) and then that $\nabla u_\e \to \nabla u$ strongly in  $L^p(\Om)$.
\par
Introduce the function $s_\e : \Om \to \{-1, 1\}$ defined by 
$$
s_\e(x):=
\begin{cases}
 1
& \text{ in }\Om^+_\e \,, \\
-1
& \text{ in }\Om^-_\e \,, \\
\end{cases}
$$
where 
$\dsp \Om^+_\e:=\{\det \nabla u_\e \geq 0 \}$ and $\dsp\Om^-_\e:=\{\det \nabla u_\e < 0 \}$. 
Remark that $s_\e \in BV(\Om;\{\pm1\})$ and, 
by \eqref{equibounded} and Lemma \ref{costo-inversione},
$J_{s_\e}$ is the union of $C/\eps^{N-1}$ facets of $(N{-}1)$-dimensional measure of order $\eps^{N-1}$, whence
\[
\H^{N-1}(J_{s_\e}) < C \quad \text{uniformly in}\ \e\,.
\]
Therefore, applying standard compactness results for sets of finite perimeter (see \cite{AFP}), 
we can extract a subsequence $s_\e$ (not relabelled) converging to a function $s\in BV(\Om;\{\pm1\})$ strongly in $L^1(\Om)$. 
Let $\bar x\in \Om$ be a Lebesgue point for both $\nabla u$ and $s$ and let $B_r(\bar x)$ denote the 
ball of radius $r$ and centre $\bar x$. 
Assume that the Lebesgue value of $s$ at $\bar x$ is 1. Then, by Theorem \ref{thm-rigidity}, 
\eqref{stimadist}, and the fact 
that the maximum distance between matrices in $O(N)$ and $SO(N)$ is bounded, one finds
\be\label{conti}
\begin{split}
&\fint_{B_r(\bar x)} |\nabla u_\e - R_\e^r H|^p \, \d x \leq
C \fint_{B_r(\bar x)} \dist^p(\nabla u_\e,SO(N)H) \, \d x \\
& \leq C \fint_{B_r(\bar x)} \big( \dist^p(\nabla u_\e,O(N)H)  + |s_\e(x)  - 1| \big)    \,\d x 
 \leq C r^{-N}\e   + C \fint_{B_r(\bar x)} |s_\e(x)  - 1|     \,\d x \,,
\end{split}
\ee
for some $R_\e^r\in SO(N)$. 
Using the strong convergence of $s_\e$ to $s$ in $L^1$, up to 
extracting a further subsequence, 
one can pass to the limit as $\e\to 0$ in \eqref{conti} and get
\be\label{quasi}
\fint_{B_r(\bar x)} |\nabla u - R^r H|^p \, \d x
\leq 
C  \fint_{B_r(\bar x)} | s(x)  - 1|     \,\d x \,,
\ee
where $R^r\in SO(N)$.
Letting $r\to 0$ in \eqref{quasi}, and possibly extracting a further subsequence, 
we deduce that the Lebesgue value of $\nabla u$ at $\bar x$ is 
$RH$ for some $R\in SO(N)$. Therefore the  Lebesgue value of $\nabla u$ 
at every Lebesgue point where $\tilde s=1$ (where $\tilde s$ is the Lebesgue representative of $s$) 
is an element of $SO(N)H$. 
We apply the same argument to $Q u_\e$, where $Q\in O(N){\setminus} SO(N)$ is any fixed reflexion, to 
find that 
the  Lebesgue value of $\nabla u$ 
at every Lebesgue point where $\tilde s=-1$ is an element of $(O(N){\setminus} SO(N))H$.
Moreover the set $\{\nabla u \in SO(N)H\}$ is of finite perimeter in $\Om$, since $s\in BV(\Om;\{\pm1\})$.
\par
In order to show the strong convergence, we will show that the $L^p$ norm is conserved, namely, 
$$
\lim_{\e \to 0^+} \int_\Om |\nabla u_\e|^p \, \d x =  \int_\Om |\nabla u|^p \, \d x \,.
$$
Fix $\eta > 0$ and let $\Om_\e^\eta := \big\{\dist(\nabla u_\e,O(N)H) >\eta \big\}$. 
Since $| \Om_\e^\eta| \to 0$ in measure, one has that 
\bes
 \int_{\Om_\e^\eta} |\nabla u_\e|^p \, \d x \leq  C \int_{\Om_\e^\eta}  \dist^p(\nabla u_\e,O(N)H) \, \d x 
 + C | \Om_\e^\eta| \leq 
 C \big(\e E_\e + | \Om_\e^\eta| \big) \to 0 \,.
\ees
Then 
\bes
\begin{split}
 \int_{\Om} |\nabla u_\e|^p \, \d x &= 
 \int_{\Om_\e^\eta} |\nabla u_\e|^p \, \d x + 
  \int_{\Om\setminus \Om_\e^\eta} |\nabla u_\e|^p \, \d x
  = 
  o(1) + (N^{p/2} + \sigma (\eta))  | \Om\setminus \Om_\e^\eta | \\
  & \to   (N^{p/2}+ \sigma (\eta)) =
  \int_{\Om} |\nabla u|^p \, \d x + \sigma(\eta) \,,
\end{split}
\ees
with $\sigma(\eta) \to 0$ as $\eta \to 0^+$.
\par
Finally, we prove that
$u$ is a collection of an at most countable family of rigid 
deformations.
To this end, fix $Q\in O(N)\setminus SO(N)$ and  define
\[
 v:= 
\begin{cases}
u  & \text{if}\ \nabla u\in SO(N) H \,, \\
Qu & \text{if}\ \nabla u\in (O(N){\setminus} SO(N))H \,.
\end{cases}
\]
Since $v\in SBV(\Om;\R^N)$ and $\nabla v\in SO(N)H$ a.e.\ in $\Om$, we can appeal to \cite[Theorem 1.1]{CGP} and deduce that 
there exists a Caccioppoli partition $\{E_i\}_{i\in\N}$ subordinated to $J_v$,
such that  
$$
v(x) = \sum_{i\in\N} (R_i H x + b_i)\chi_{E_i}(x)\,,
$$
where $R_i\in SO(N)$ and $b_i\in\R^N$. 
Taking into account that $\{\nabla u \in SO(N)H\}$ has finite perimeter,
this implies \eqref{grad-SBV} and \eqref{pw-rigidity}.
The last part of the statement follows from Lemma \ref{connessioni} below.
\end{proof}
\begin{lemma}\label{connessioni} 
Let $\Om\subset \R^N$ be an open, bounded set with Lipschitz boundary, and let
$u\in W^{1,\infty}(\Om;\R^N)$. 
Suppose that there exists a Caccioppoli partition $\{E_i\}_{i\in\N}$ of $\Om$ 
such that 
\bes
u(x) = \sum_{i\in\N} (P_i x + b_i)\chi_{E_i}(x)\,,
\ees
where $P_i\in \Mnn$ and $b_i\in\R^N$. 
If $\partial^* E_i \cap \partial^* E_j \neq \emptyset$, then 
$\rank (P_i - P_j)\leq 1$ and,
denoted by $\nu_i$ the inner normal to $E_i$,
$\nu_i$ is constant on $\partial^* E_i \cap \partial^* E_j$.
%
\end{lemma}
\begin{proof}
Let $\bar{x}\in \partial^* E_i \cap \partial^* E_j$. Then
$\bar{x}\in J_{\nabla u}$, i.e., 
there exists a unit vector $\nu\in \R^{N}$ such that 
\be\label{densita1/2}
\lim_{\e\to 0^+} \fint_{B_\e^{\nu+}(\bar x)} | \nabla v - P_i  |  \, \d x  =0 \,,
\quad
\lim_{\e\to 0^+} \fint_{B_\e^{\nu-}(\bar x)} | \nabla v - P_j  |  \, \d x  =0 \,,
\ee
where 
$B_\e^{\nu\pm}(\bar x) := B_\e(\bar x)\cap \{x\colon  \pm x\cdot \nu >0\} $ 
and 
$B_r(p)$ denotes the ball of radius $r$ and centre $p$;
cf.\ \cite[Definition 3.67 and Example 3.68]{AFP}.
For $x\in B_1(0)$ define the sequence $v_\e (x):= \frac{1}{\e} v(\e(x-\bar x))$. 
Then $\nabla v_\e (x)= \nabla v(\e (x-\bar x))$ and, by  \eqref{densita1/2}, 
we have that $\nabla v_\e \to \chi P_i + (1-\chi)P_j$ in $L^p(B_1(0); \Mnn)$ for every $p\geq 1$, 
where $\chi$ is the characteristic function of $B_1^{\nu+}(0)$.
The thesis now follows from the rigidity of the two-gradient problem.
\end{proof}
%
%
%
%

\subsection{Lower and upper bounds}
In this section we provide lower and upper bounds for the $\Gamma$-limit of (any subsequence of)  $\{E_\e\}$ in terms of 
interfacial energies that penalise changes of orientation.
In what follows we denote by $E'(\cdot)$ and $E''(\cdot)$ the $\Gamma$-$\liminf$ and the $\Gamma$-$\limsup$ as $\e\to0^+$, respectively, of the sequence $\{E_\e\}$ with respect to the strong convergence in $W^{1,p}(\Omega;\R^N)$. We also introduce a ``localised'' version of the functionals $E_\e$ by setting, for any open set $A\subset\R^N$,
$$
E_\e(u;A)= \e^{N-1}\sum_{\xi\in B_1\cup B_2} \!\!\!\sum_{\substack{x\in\e\Z^N\cap A\\x{+}\e\xi\in\e\Z^N\cap A}}\!\!\! \Big| \frac{|u(x{+}\e\xi){-}u(x)|}{\e}-|H \xi| \Big|^p \,.
$$

Moreover, given $ \nu$ in the unit sphere $S^{N-1}$, we denote by  $Q_\nu$ any cube centred at $0$, with side length $1$ and two faces orthogonal to $\nu$, and by $u_0^\nu$ 
the piecewise affine function defined by 
 \begin{eqnarray*}
u_0^\nu(x):=\begin{cases}
Hx & \text{if}\ \langle x,\nu\rangle\geq 0 \,,\\
R_\nu x & \text{otherwise,}
\end{cases}
\end{eqnarray*}
where $R_\nu\in (O(N){\setminus} SO(N))H$ is such that $H-R_\nu=a\otimes\nu$ for some $a\in\R^N$.
\begin{proposition}[Lower bound]\label{propo:lower}
For every $u\in W^{1,\infty}(\Omega;\R^N)$ with $\nabla u\in SBV(\Omega; O(N)H)$, one has
\begin{eqnarray*}
E'(u)\geq  \int_{J_{\nabla u}} g_1(\nu_{\nabla u})\, \d{\mathcal H}^{N-1} \,,
\end{eqnarray*}
where $g_1: S^{N-1}\to [0,+\infty)$ is defined by 
\begin{eqnarray*}
g_1(\nu):=\inf\Big\{\liminf_{n\to\infty} E_{\e_n}(u_n; Q_\nu)\colon \e_n\to 0 \,,\ u_n\to u_0^\nu\ \hbox{strongly in}\ W^{1,p}(Q_\nu;\R^N)\Big\}\,,
\end{eqnarray*}
and satisfies $g_1(\nu)=g_1(-\nu)$.
\end{proposition} 
\begin{proof}
Suppose that $u_\e\to u$ in $W^{1,p}(Q_\nu;\R^N)$ and $\sup_\e E_\e(u_\e)<+\infty$. Let 
\begin{equation}\label{maxlength}
r:=\sup\{|\xi| \colon \xi\in B_1\cup B_2\}\,,
\end{equation}
and define the family of positive measures 
$$
\mu_\e:=\sum_{x\in\LL_\e^r}\Bigg(\sum_{\xi\in B_1\cup B_2}\Big| \frac{|u_\e(x{+}\e\xi){-}u_\e(x)|}{\e}-|H \xi| \Big|^p \Bigg) \delta_x \,,
$$
where $\LL_\e^r=\{x\in\LL_\e \colon \dist(x,\partial\Omega)\geq r\e\}$. Note that
$$
E_\e(u_\e)\geq \mu_\e(\Omega) \,,
$$
hence, up to passing to a subsequence, we may suppose that there exists a positive measure $\mu$ such that
$\mu_\e\stackrel *\weak  \mu$. We now use a blow-up argument. By the Radon-Nikodym Theorem, we 
can decompose $\mu$ into two mutually singular positive measures:
$$
\mu = g\,{\mathcal H}^{N-1}\res J_{\nabla u}+\mu^s \,.
$$
We complete the proof if we show that
$$
g(x_0)\geq g_1(\nu_{\nabla u}(x_0))\quad \hbox{for } {\mathcal H}^{N-1}\hbox{-a.e.}\ x_0\in J_{\nabla u} \,.
$$
By the properties of $BV$ functions we know that for ${\mathcal H}^{N-1}$-a.e.\ $x_0\in J_{\nabla u}$
\begin{itemize}
\item[(i)]$\displaystyle\lim_{\rho\to 0}\frac{1}{\rho^{N-1}}{\mathcal H}^{N-1}(J_{\nabla u}\cap(x_0+\rho Q_{\nu_{\nabla u}(x_0)}))=1$\,,
\item[(ii)] $\displaystyle\lim_{\rho\to 0}\frac{1}{\rho^N}\int_{\rho Q^{\pm}_{\nu_{\nabla u}(x_0)}}|\nabla u(y)-\nabla u^{\pm}(x_0)|\, \d x=0$\,,
\item[(iii)] $g(x_0)=\displaystyle\lim_{\rho\to 0}\frac{\mu(x_0+\rho Q_{\nu_{\nabla u}(x_0)})}{{\mathcal H}^{N-1}(J_{\nabla u}\cap(x_0+\rho Q_{\nu_{\nabla u}(x_0)}))}$\,.
\end{itemize}
Fix such a point $x_0\in J_{\nabla u}$ and let $\{\rho_n\}$ be a sequence of positive numbers converging to zero such that $\mu(x_0+\rho_n \partial Q_{\nu_{\nabla u}(x_0)})=0$.  From (i) and (iii) it follows that 
$$
g(x_0)=\lim_{n\to\infty}\lim_{\e\to 0}\frac{1}{\rho_n^{N-1}}\mu_\e (x_0+\rho_n Q_{\nu_{\nabla u}(x_0)}) \,.
$$
Note that for every $\e>0$ and $n\in\N$ we can find $\rho_{n,\e}$ and   $x_{0,\e}\in\e\Z^N$ such that  
$\lim_{\e\to 0}\rho_{n,\e}=\rho_n$,
$\lim_{\e\to 0} x_{0,\e}=x_0$, and
$$
x_0+ (\rho_n {-}R\e)\, Q_{\nu_{\nabla u}(x_0)}\supseteq x_{0,\e}+\rho_{n,\e} Q_{\nu_{\nabla u}(x_0)} \,.
$$
Then
$$
g(x_0)\geq\lim_{n\to\infty}\lim_{\e\to 0}\frac{1}{\rho_{n,\e}^{N-1}} E_\e(u_\e; x_{0,\e}+\rho_{n,\e} Q_{\nu_{\nabla u}(x_0)}) \,.
$$
Set 
\begin{align*}
 u_{n,\e}(x)&:=u_\e(x_{0,\e}{+}\rho_{n,\e} x) \quad  \text{for}\ x\in \frac{\e}{\rho_{n,\e}}\Z^N \,, \\
 v_{n,\e}&:=u_{n,\e}-c_{n,\e} \,,\\
 F^{\pm}&:=\nabla u^{\pm}(x_0) \,.
\end{align*}
Since $u_\e\to u$ in $W^{1,p}(\Omega;\R^N)$, from (ii) we deduce that  there exist constants $c_{n,\e}$ such that 
$$
\lim_{n\to \infty}\lim_{\e\to 0}\|v_{n,\e}-u_{F^+,F^-}\|_{W^{1,p}(Q_{\nu_{\nabla u}(x_0)};\R^N)}=0 \,,
$$
where
\begin{eqnarray*}
u_{F^+,F^-}(x):=\begin{cases}
F^+x & \text{if}\ \langle x,\nu\rangle\geq 0 \,,\\
F^-x & \text{otherwise.}
\end{cases}
\end{eqnarray*}
Using a standard diagonalisation  argument and the translational invariance of $E_\e$ with respect to both independent and dependent variables,
we can find a sequence of positive numbers $\lambda_k\to 0$ and a sequence $v_k$ converging to $u_{F^+,F^-}$ in $W^{1,p}(Q_{\nu_{\nabla u}(x_0)};\R^N)$ such that
$$
g(x_0)\geq\lim_{k\to\infty} E_{\lambda_k}(v_k; Q_{\nu_{\nabla u}(x_0)}) \,.
$$
The conclusion then follows by the very definition of $g_1$ and taking into account that, by invariance with respect to $O(N)$, we may replace $v_k$ by $R v_k$, $R\in 
O(N)$, without changing the energy.
\end{proof}
\begin{remark}\label{boundpositive}
By a slicing argument, we may show that 
\begin{eqnarray}\label{stimainf}
\inf_{\nu\in S^{N-1}}g_1(\nu)\geq \frac{C_0}{N}>0 \,,
\end{eqnarray}
where $C_0$ is as in Lemma \ref{costo-inversione}.
This implies in particular that
$$
E'(u)\geq \frac{C_0}{N} \, {\mathcal H}^{N-1}(J_{\nabla u}) \,.
$$
In order to prove \eqref{stimainf}, let us set, for every $k\in\{1,\dots, N\}$,
$$
S_\nu^k:=\Pi^{e_k} (Q_\nu\cap\{\langle x,\nu\rangle=0\}) \,,
$$ 
where $\Pi^{e_k}$ denote the orthogonal projection on $\{\langle x,e_k\rangle=0\}$.
 Let $\e_n\to 0$ and $u_n\to u_0^\nu$ strongly in $W^{1,p}(Q_\nu;\R^N)$, and for any $k\in\{1,\dots,N\}$ set
 $$
 I_n^k:=I_n^{k+}\cup I_n^{k-}\,,
 $$
 where
 $$
I_n^{k\pm}:= \{i\in\e\Z^N\cap S_\nu^k \colon \pm\det \nabla u_n(x)>0\ \forall\, x\in ((\Pi^{e_k})^{-1}(i)+[0,\e_n]^N)\cap Q_\nu\} \,.
$$
Then, one easily gets that 
$$
\int_{Q_\nu} |\nabla u_n-\nabla u_0^\nu|^p\, \d x\geq C\e^{N-1}\# I_n^k+o(1) \,,
$$
so that $\e^{N-1}\# I_n^k\to 0$. Hence, 
by Lemma \ref{costo-inversione} we deduce that
$$
\liminf_{n\to\infty} E_{\e_n}(u_n, Q_\nu)\geq  C_0 \max_{k=1,\dots, N}\liminf_{n\to\infty} \e_n^{N-1}\# (I_n^k)^c=C_0 \max_{k=1,\dots, N} |S_\nu^k|\geq \frac{C_0}{N} \,.
$$

\end{remark}
We now provide an upper estimate of $E''(u)$ for a suitable subclass of the limiting deformations $u$ identified by Theorem \ref{thm:compactness}. We say that a set $K\subset\Omega$ is polyhedral with respect to $\Omega$ if it consists of the intersection of $\Omega$ with the union of a finite number of $(N{-}1)$-dimensional simplices of $\R^N$. We set then
$$
{\mathcal W}(\Omega):=\{u\in W^{1,\infty}(\Omega;\R^N) \colon \nabla u\in SBV(\Omega; O(N)H)\ \hbox{and}\ J_{\nabla u}\ \hbox{is polyhedral in}\ \Omega\} \,.
$$	
\begin{proposition}[Upper bound]\label{propo:up}
For any $u\in {\mathcal W}(\Omega)$, the following inequality holds true:
$$
E''(u)\leq \int_{J_{\nabla u}} g_2(\nu_{\nabla u})\, \d{\mathcal H}^{N-1} \,,
$$
where, $g_2:S^{N-1}\to [0,+\infty)$ is defined by
\begin{eqnarray*}
g_2(\nu):=\lim_{T\to \infty}\frac{1}{T^{N-1}}\inf\{E_{1}(u; TQ_\nu) \colon u(x)=u_0^\nu(x)\ \hbox{if}\  \dist(x,\partial (TQ_\nu))\leq r\} \,,
\end{eqnarray*}
and $r$ is given by \eqref{maxlength}.
\end{proposition}
\begin{proof}
We  analyse only the case where $J_{\nabla u}$ is the restriction to $\Omega$ of a hyperplane, since the case of a general polyhedral boundary is easily recovered by a gluing argument. Fix then $\nu\in S^{N-1}$,  let $J_{\nabla u}=\Pi_\nu\cap \Omega$, where $\Pi_\nu$ is a hyperplane orthogonal to $\nu$. By translational  and rotational invariance, without loss of generality we may assume that $\Pi_\nu=\{x\in\R^N \colon \langle x,\nu\rangle=0\}$ and $u=u_0^\nu$. Given $\delta >0$, let $T_\delta>0$ and  $u_\delta$ such that $u_\delta (x)=u_0^\nu(x)$ if $ \dist(x,\partial (T_\delta Q_\nu))\leq r$ and 
$$
\frac{1}{T_\delta^{N-1}}E_{1}(u_\delta; T_\delta Q_\nu)\leq g_2(\nu)+\delta \,.
$$
Let $\{b_1,\dots, b_{N}\}$ be an orthonormal base of $\R^N$ such that $b_N=\nu$ and  $Q_\nu=\{x\in\R^N \colon |\langle x,b_l\rangle |<1/2 \,,\ l=1,\dots, N\}$. For any $j\in \bigoplus_{l=1}^{N-1}\Z T_\delta b_l$ set $x_j= [j]=([j_1],\dots,[j_{N}])$,
where $[\cdot]$ denotes the integer part.
Then let $u_\e \colon \e\Z^N\cap \Omega\to \R^N$ be such that
$$
u_\e(\e i  )=u_\delta (i-x_j)+\e Hx_j\quad \text{for}\ i\in\Z^N\cap T_\delta Q_\nu+x_j \,,\ j\in \bigoplus_{l=1}^{N-1}\Z(T_\delta+2r)b_l \,,
$$
and $u\equiv u_0^\nu$ otherwise. Then $u_\e\to u$ strongly in $W^{1,p}(\Omega;\R^N)$ and
$$
E''(u) \leq \limsup_{\e\to 0} E_\e(u_\e)\leq \frac{{\mathcal H}^{N-1}(J_{\nabla u})}{T_\delta^{N-1}} E_1(u_\delta; T_\delta Q_\nu)+O(\delta)\leq (g_2(\nu)+\delta){\mathcal H}^{N-1}(J_{\nabla u})+O(\delta) \,.
$$
The conclusion follows by letting $\delta$ tend to $0$.
\end{proof}
\begin{remark}
Testing the infimum problems defining $g_2$ with $u=u_{F,G}$, we easily get that $g_2(F,G,\nu)\leq C$ uniformly in $(F,G,\nu)$. In particular, we have that for every $u\in {\mathcal W}(\Omega)$
$$
E''(u)\leq C {\mathcal H}^{N-1}(J_{\nabla u}) \,.
$$
\end{remark}
\begin{remark}
The computation of the $\Gamma$-limit of $E_\e$ remains an open question. However, we believe that the $bounds$ provided in Propositions \ref{propo:lower} and \ref{propo:up} give some insight into its derivation.  Assume that the following result holds true: given any test sequence $u_n$ in the definition  of $g_1$, there exists a sequence of functions $v_n$ such that $v_n\to u_0^\nu$ strongly in $W^{1,p}(Q_\nu,\R^N)$, $v_n(x)=u_0^\nu(x)$ if $ \dist(x,\partial (T_\delta Q_\nu))\leq r$ and $E_{\e_n}(v_n;Q_\nu)\leq E_{\e_n}(u_n;Q_\nu)+o(1)$. Then it could be easily shown that $g_1=g_2$ and, consequently, the interfacial energies in Propositions \ref{propo:lower} and \ref{propo:up} would provide the $\Gamma$-limit of $E_\e(u)$ for any $u\in{\mathcal W}(\Omega)$.
\end{remark}
\section{Application to dimension reduction in nanowires}\label{sec:not}
In the present section we show an application of Theorem \ref{thm:compactness}
to the dimension reduction of a discrete model for heterogeneous nanowires.
This model was first studied in \cite{LPS,LPS2} under the assumption that 
the admissible deformations satisfy the non-interpenetration condition. 
Here 
we remove such assumption, and 
we  show that, by incorporating into the energy the effect of interactions in a certain finite range,
one can recover the  results of \cite{LPS,LPS2} and get even further insight into the problem. 
\par
Let $L>0$, $k\in \N$, $\Om_{k\e} : = (-L,L)\times (-k\e, k\e )^{N-1}$. 
We consider the discrete thin domain $\L_\e \subset \R^N$ defined as
\be\label{nano}
\LL_\e(k) :=
\e\Z^N \cap \overline\Om_{k\e}\,, 
\ee
where $\overline\Om_{k\e}$ is the union of all hypercubes with vertices in $\e\Z^N$ that have 
non-empty intersection with $\Om_{k\e}$.
In the physically relevant case of $N=3$,  the set $\L_\e(k)$ models the crystal structure of a 
nanowire of length $2L$
and thickness $2k \e $, where 
$k$ is the number of parallel atomic planes.
Nonetheless we will state all the results for a general $N$, since their proof does 
not depend on the dimension. 
Notice that in definition \eqref{nano} the dependence on $k$ is explicit; this parameter will indeed  
play a major role in the subsequent analysis.
The bonds between the atoms are defined by means of the sets $B_1$ and $B_2$ 
(see \eqref{bonds1}--\eqref{bonds2}) exactly as in the  previous section.
\par
We assume that $\L_\e$ is composed of two species of atoms, 
occupying the points contained in the subsets
\begin{align*}
\L_\e^-(k) &:= \{x \in\L_\e(k)\colon x_1<0\} \,,\\
\L_\e^+(k) &:= \{x \in\L_\e(k)\colon x_1\ge0\} \,,
\end{align*}
respectively, where $x=(x_1,\dots,x_N)$.
The two species of atoms are characterised by equilibrium distances 
given by $\e$ and $\lambda\e$, respectively, where $\lambda\in(0,1]$ is fixed;
the case $\lambda\in(0,1)$ models a heterogeneous nanowire, while the case $\lambda=1$ refers to a homogeneous nanowire. 
Specifically, the total interaction energy relative to a deformation 
$u \colon\L_{\e}(k)\to\R^N$ is defined as
\be\label{eng-eps}
\E\unolambda_{\e}(u,k) :=\!\!\!\!\!
 \sum_{\substack{
x\in \L^-_{\e}(k)\\  
\xi\in B_1\cup B_2 \\
x + \e \xi \in \L_\e(k)
 }}
\!\!\!\! \!
c(\xi)
\left|\frac{|u(x+\e \xi)-u(x)|}{\e}-|H \xi|\right|^p 
+ 
\!\!\!\!\!
\sum_{\substack{
x\in \L^+_{\e}(k)\\  
\xi\in B_1\cup B_2 \\
x + \e \xi \in \L_\e(k)
 }}
 \!\!\!\!\!
c(\xi)
\left|\frac{|u(x+\e \xi)-u(x)|}{\e}-\lambda |H \xi|\right|^p
\ee
where $H\in GL^+(N)$ and the coefficient $c(\xi)$ is equal to some $c_1>0$ for $\xi\in B_1$ and to $c_2>0$ for $\xi\in B_2$.
In this section we restrict for simplicity to $p$-harmonic potentials; however, our analysis can be generalised to potentials as those appearing in \eqref{eng-general}.
In principle, all the results that we present in the sequel extend to the case when the two components of the nanowire 
have equilibria of the form $H^-$ and 
$H^+$ where  $H^-, H^+\in GL^+(N)$. 
We have chosen to analyse the case when $H^+=\lambda H^-$, since this  is particularly meaningful 
in applications. 
\par
We study the limit behaviour of 
$\E\unolambda_{\e}(\cdot,k)$ as $\eps\to 0^+$, thus performing simultaneously
a discrete to continuum limit and a dimension reduction to a one-dimensional system.
The limit functional was derived in \cite{LPS,LPS2} by means of $\Gamma$-convergence, under the assumption that 
the admissible deformations fulfil the non-interpenetration condition, namely, that
the Jacobian determinant of (the piecewise affine interpolation of) any deformation is strictly positive almost everywhere.
The non-interpenetration assumption was used in several parts of the analysis; in particular,
it was needed to prove that the limit functional (dependent on $k$) scales like $k^N$ as $k\to\infty$.
\par
The main novelty of the present paper is that we remove the non-interpenetration assumption made 
in \cite{LPS,LPS2}, 
allowing for changes of orientations. Furthermore, in the study of the $\Gamma$-limit we define a stronger topology that accounts for such changes.  
In the proof of the new results, only those parts that differ from \cite{LPS,LPS2} will be shown in 
details.
We remark that,  in dimension two, our analysis corresponds to the first-order $\Ga$-limit of a functional of the kind studied in \cite{ABC,Schm08} without non-interpenetration assumptions.
\par
In the sequel of the paper we will often consider the rescaled domain $\frac{1}{\eps}\Om_{k\eps}$,
which converges, as $\eps\to0^+$, to the unbounded strip
\bes
\Om_{k,\infty}:= \R\times (-k, k)^{N-1} \,.
\ees
We define the associated lattice and subsets
\begin{align*}
\L_{\infty}(k)& :=\Z^N\cap\overline\Om_{k,\infty} \,,\\
\L_{\infty}^-(k)    &:= \{x \in\L_\infty(k)\colon x_1<0\} \,,\\
\L_{\infty}^+(k)   &:= \{x \in\L_\infty(k)\colon x_1\ge0\} \,,
\end{align*}
where $\overline\Om_{k,\infty}$ is the union of all hypercubes with vertices in $\Z^N$ that have 
non-empty intersection with $\Om_{k,\infty}$.
For $u\colon \L_{\infty}(k) \to \R^N$ we define
\begin{equation}\label{eng-infty}
\E\unolambda_{\infty}(u,k) := 
\!\!\!\!\!
 \sum_{\substack{
x\in \L^-_{\infty}(k)\\  
\xi\in B_1\cup B_2 \\
x +  \xi \in \L_\infty(k)
 }}
\!\!\!\! \!
c(\xi)
\Big||u(x+ \xi)-u(x)|-|H \xi|\Big|^p 
+ 
\!\!\!\!\!
\sum_{\substack{
x\in \L^+_{\infty}(k)\\  
\xi\in B_1\cup B_2 \\
x + \xi \in \L_\infty(k)
 }}
 \!\!\!\!\!
c(\xi)
\Big||u(x+ \xi)-u(x)|-\lambda |H \xi|\Big|^p.
\end{equation}
We identify every deformation $u$ of the lattice $\L_{\eps}(k)$ 
by its piecewise affine interpolation with respect to the triangulation $\e\T$. 
By a slight abuse of notation, such extension is still 
denoted by $u$. 
We can then define the domain of the functional \eqref{eng-eps} as
\bes
\begin{split}
\A_{\eps}(\Om_{k\eps}):= \big\{ u\in C^0(\overline\Om_{k\eps};\R^N) \colon & u \ 
\text{piecewise affine,}\\ 
& \D u \  \text{constant on}\ \Om_{k\eps}\cap \e T\  \forall\, T\in\T \big\} \,.
\end{split}
\ees
Similarly, for \eqref{eng-infty} we define 
\bes
\begin{split}
\A_{\infty}(\Om_{k,\infty}):= \big\{ u\in C^0(\overline\Om_{k,\infty};\R^N) \colon & u \ \text{piecewise affine,}\\ 
& \D u \ \text{constant on}\ \Om_{k,\infty}\cap T\ \forall\, T\in\T \big\} \,.
\end{split}
\ees
\par
As customary in dimension reduction problems, we 
rescale the domain $\Om_{k\eps}$ to a fixed domain $\Om_k$, independent of $\eps$,
by introducing the change of variables $z(x):=(x_1, \e x_2,\dots, \e x_N)$. 
Accordingly,  
for each 
$u\in\A_{\eps}(\Om_{k\eps})$ 
we define
$\ut (x):= u(z(x))$. 
Moreover we set
$\Om_{k}:= A_\e^{-1}(\Om_{k\eps})= (-L,L)\times (-k,k)^{N-1}$, 
where $A_\e\in\Mnn$ is the diagonal matrix 
\be\label{cambio-var}
A_\e :=\diag(1, \e,\dots,\e);
\ee
i.e., $z(x) = A_\e x$.
In this way we can recast the functionals \eqref{eng-eps} defined over varying domains 
into functionals defined on deformations of the fixed domain $\Om_k$. Precisely we set
\begin{equation}\label{funzionale-risc}
\I\unolambda_{\e}(\tilde u,k) := \E\unolambda_\e(u,k) \quad \text{for}\ \ut 
\in \widetilde\A_{\e}(\Om_{k}) \,,
\end{equation}
with 
\bes
\begin{split}
\widetilde\A_{\eps}(\Om_{k}):= \big\{ \ut\in C^0(A_\e^{-1}(\overline\Om_{k\e});\R^N) \colon & \ut \ \text{piecewise affine,}\\ 
& \D \ut \ \text{constant on}\ \Om_{k}\cap (A_\eps^{-1}\e T)\ \forall\, T\in\T \big\} \,.
\end{split}
\ees
For later use it will be convenient to set the following notation:
$$
\Om_k^- := (-L,0){\times} (-k,k)^{N-1}\,, \quad  \Om_k^+ := (0,L){\times} (-k,k)^{N-1}\,.
$$
\subsection{Definition and properties of minimal energies}
Throughout the paper, $I$ is the identity matrix and $J$ is the reflection matrix such that $Je_1=-e_1$ and
$Je_i=e_i$ for $i=2,\dots,N$.
\par
We will study the $\Gamma$-limit of the sequence $\I\unolambda_{\e}(\cdot,k)$ as $\e\to 0^+$ for every fixed $k$.
For this purpose we introduce the quantity $\gamma(P_1,P_2;k)$ for $P_1,P_2\in O(N)\cup \lambda\, O(N)$, 
which represents the 
minimum cost of a transition from a well to another. Specifically, for each 
$P_1 \in O(N)$ and $P_2\in \lambda\, O(N)$ we define
\begin{subequations} \label{gamma}
\begin{equation}\label{gamma1}
\begin{split}
\gamma(P_1,P_2;k):=\inf\big\{ &\E\unolambda_{\infty}(v,k) \colon M>0\,,
\  v\in \A_{\infty}(\Om_{k,\infty})\,, \\
& \D v=P_1H \ \text{for} \ x_1\in(-\infty,-M)\,,\, \ \D v= P_2 H \ \text{for} \ x_1\in(M,+\infty) 
\big\}\,;
\end{split}
\end{equation}
for $P_1,P_2 \in O(N)$
\begin{equation}\label{gamma2}
\begin{split}
\gamma(P_1,P_2;k):=\inf\big\{ & \E\unouno_{\infty}(v,k) \colon M>0\,,
\  v\in \A_{\infty}(\Om_{k,\infty})\,, \\
& \D v=P_1H \ \text{for} \ x_1\in(-\infty,-M)\,,\, \ \D v= P_2 H \ \text{for} \ x_1\in(M,+\infty) 
\big\}\,,
\end{split}
\end{equation}
where 
$$
\E\unouno_{\infty}(v,k) : = 
\sum_{\substack{
x\in \L_{\infty}(k)\\  
\xi\in B_1\cup B_2 \\
x +  \xi \in \L_\infty(k)
}}
\!\!\!\! \!
c(\xi)
\Big||v(x+ \xi)-v(x)|-|H \xi|\Big|^p \,;
$$
for $P_1,P_2 \in \lambda\, O(N)$
\begin{equation}\label{gamma3}
\begin{split}
\gamma(P_1,P_2;k):=\inf\big\{ & \E\lambdalambda_{\infty}(v,k) \colon M>0\,,
\  v\in \A_{\infty}(\Om_{k,\infty})\,, \\
& \D v=P_1H \ \text{for} \ x_1\in(-\infty,-M)\,,\, \ \D v= P_2 H \ \text{for} \ x_1\in(M,+\infty) 
\big\}\,,
\end{split}
\end{equation}
where 
$$
\E\lambdalambda_{\infty}(v,k) : = 
\sum_{\substack{
x\in \L_{\infty}(k)\\  
\xi\in B_1\cup B_2 \\
x +  \xi \in \L_\infty(k)
}}
\!\!\!\! \!
c(\xi)
\Big||v(x+ \xi)-v(x)|- \lambda |H \xi|\Big|^p \,.
$$
\end{subequations}
\par
The next proposition shows that the relevant quantities defined through \eqref{gamma} are in fact four: two estimate the cost of the transition at the interface between the energy wells $O(N)$ and $\lambda\, O(N)$, see \eqref{24a} and \eqref{24b}; one for the transition between $SO(N)$ and $O(N){\setminus}SO(N)$, see \eqref{24c}; one for the transition between $\lambda\, SO(N)$ and $\lambda\, O(N){\setminus}SO(N)$, see \eqref{24d}.
Moreover, the constants in \eqref{24c} and in \eqref{24d} are related by the proportionality rule \eqref{24e}.
\begin{proposition}\label{invariance}
For each $k\in\N$, the function $\gamma$ satisfies
for every $R,R'\in SO(N)$ and $Q,Q'\in O(N){\setminus}SO(N)$
\begin{subequations} \label{24}
\begin{alignat}{6}
\gamma(R, R';k) =&&&\ \gamma(Q, Q';k) &&=\ga(\lambda R,\lambda R';k)=\ga(\lambda Q,\lambda Q';k)= 0 \,, 
\\
\label{24a}
\gamma(R,\lambda R';k) =&&&\ \gamma(Q,\lambda Q';k) &&= \gamma(I,\lambda I;k) \,, 
\\
\label{24b}
\gamma(R,\lambda Q;k) =&&&\ \gamma(Q,\lambda R;k) &&= \gamma(I, \lambda J;k) \,, 
\\
\label{24c}
\gamma(R,Q;k) =&&&\ \gamma(Q,R;k) &&= \gamma(I,J;k) \,, 
\\
\label{24d}
\gamma(\lambda R,\lambda Q;k) =&&&\ \gamma(\lambda Q,\lambda R;k) &&= \gamma(\lambda I,\lambda J;k) \,. 
\end{alignat}
\end{subequations}
Moreover, 
\be\label{24e}
\gamma(\lambda P_1,\lambda P_2;k)
= \lambda^p \gamma(P_1,P_2;k) \quad
\text{for every}\ P_1,P_2 \in O(N) \,.
\ee
\end{proposition}
\begin{proof}
First one notices that $\gamma(P_1,P_2;k)=\gamma(JP_1,JP_2;k)$. 
Hence, the proof of \eqref{24} relies on the construction of low energy transitions between two given rotations or two given reflections,
see \cite[Proposition 2.4]{LPS}. 
Finally, standard comparison arguments yield \eqref{24e}. 
\end{proof}
We now prove estimates on the asymptotic behaviour of $\ga(I,\lambda I)$ and $\ga(I,\lambda J)$ as $k\to \infty$, which have interesting consequences towards the comparison of this model with those accounting for  dislocations in nanowires, see Section \ref{sec:disl} below.
Indeed, we show that for $\lambda\neq1$ (heterogeneous nanowire) these constants grow faster than $k^{N-1}$, while it is known that the corresponding minimum cost for nanowires with dislocations scales like $k^{N-1}$ (see discussion at the end of Section \ref{sec:disl}).
In contrast, we remark that for $\lambda=1$ one has $\ga(I,I)=0$ and $\ga(I,J)\simeq C k^{N-1}$.
Theorem \ref{thm:scaling} follows as an application of Theorem
\ref{thm:compactness}.
\begin{theorem}\label{thm:scaling}
Let $\lambda\in(0,1)$ and $(P_1,P_2)\in \{ (I,\lambda I), (I,\lambda J) \}$.
There exists $C>0$ such that
\be\label{upper-b}
\gamma(P_1,P_2;k) \leq C k^N \,.
\ee
Moreover,
\be\label{lower-b}
\lim_{k\to \infty} \frac{\gamma(P_1,P_2;k)}{k^{N-1}} = +\infty.
\ee
\end{theorem}
\begin{proof}
%
The upper bound \eqref{upper-b} is proven by comparing test functions for $\gamma(P_1,P_2;k)$ with those for 
$\gamma(P_1,P_2;1)$. Namely, let $v\in A_{\infty}(\Om_{1,\infty})$ be such that $v(x)=P_1Hx$ for every $x\in \LL_\infty^-(k)$
and $v(x)=P_2Hx$ for every $x\in \LL_\infty^+(k)$; in particular, $\nabla v= P_1H$ for $x_1\in (-\infty,-1)$ and $\nabla v=P_2 H$ for $x_1\in(0,+\infty)$.
Then one defines $u\in \A_{\infty}(\Om_{k,\infty})$ by $u(x) := k v(x/k)$, which yields 
$\gamma(P_1,P_2;k)    \leq 
\E\unolambda_{\infty}(u,k)
\leq C \, \E\unolambda_{\infty}(v,1) \,k^N$, and thus
$\gamma(P_1,P_2;k) \leq C \, \gamma(P_1,P_2;1)\,k^N$. 
Note that in the previous inequalities one uses the fact that $\nabla v\in L^{\infty}$ and that the energy of the interactions in $B_2$  can be 
bounded, using the Mean Value Theorem, by the energy of the interactions in $B_1$. 
\par
For the proof of the lower bound \eqref{lower-b} we will use Theorem \ref{thm:compactness}.
By contradiction, suppose that there exist a sequence $k_j\nearrow\infty$ and a sequence 
$\{u_j\}\subset \A_{\infty}(\Om_{k_j,\infty})$ such that 
\begin{equation}\label{infinitesimal}
\frac{1}{k_j^{N-1}}\,\E\unolambda_{\infty}(u_j,k_j) < C \,, 
\end{equation}
for some positive $C$.
Define $v_j\colon\Om_{1,\infty}\to \R^N$ as $v_j(x) := \frac{1}{k_j} u_j(k_j x)$. 
Accordingly, we consider the rescaled lattices 
$$
\L_j :=\frac{1}{k_j}\Z^N\cap\overline\Om_{1,\infty} \,,\quad
\L_j^+ : =\L_j \cap \{x_1 > 0\}\,, \quad
\L_j^- : =\L_j \cap \{x_1 < 0\}\,.
$$ 
%
Expressing $\E\unolambda_{\infty}(u_j,k_j)$ in terms of $v_j$, one finds
\begin{equation}\label{energy-rescaled}
\E\unolambda_{\infty}(u_j,k) = 
\!\!\!\!\!\!\!\!
 \sum_{\substack{
x\in \L^-_{j}\\  
\xi/ k_j  \in B_1\cup B_2 \\
x +  \xi/k_j \in \L_j
 }}
\!\!\!\!\!\!\!\!
c(\xi)
\left|
\frac{| v_j(x+   \frac{\xi}{k_j}  )-v_j(x)|}{\frac1{k_j}}
-|H \xi|\right|^p 
+ 
\!\!\!\!\!\!\!\!
\sum_{\substack{
x\in \L^+_{j}\\  
\xi/k_j \in B_1\cup B_2 \\
x +  \xi/k_j \in \L_j
 }}
\!\!\!\!\!\!\!\!
c(\xi)
\left|\frac{| v_j(x+   \frac{\xi}{k_j}  )-v_j(x)|}{\frac1{k_j}}-\lambda |H \xi|\right|^p.
\end{equation}
The above term controls the (piecewise constant) gradient of $v_j$. 
From \eqref{infinitesimal}, \eqref{energy-rescaled}, and Theorem \ref{thm:compactness} we deduce that,
up to subsequences, 
$\nabla v_j \to \nabla v$  in  $L^p((-1,1)^{N})$,
where $\nabla v\in O(N)H$ 
for a.e. $x\in (-1,0) \times (-1,1)^{N-1}$ and 
$\nabla v\in \lambda\, O(N)H$  for a.e. $x\in (0,1) \times (-1,1)^{N-1}$. 
Precisely, 
$$
v(x) = \sum_{i\in\N} (R_i H x + a_i)\chi_{E_i}(x) +  \sum_{j\in\N} (\lambda Q_j H x + b_j)\chi_{F_j}(x) \,,
$$
where $R_i,Q_j \in O(N)$, $a_i,b_j\in\R^N$, and $\{E_i\}$ (respectively  $\{F_j\}$) is a Caccioppoli partition of $(-1,0) \times (-1,1)^{N-1}$ 
(respectively  of $(0,1) \times (-1,1)^{N-1}$).
Then, since $\{E_i\} \cup \{F_j\}$ is a Caccioppoli partition of $(-1,1)^N$,
by the local structure of Caccioppoli partitions (see e.g.\ \cite[Theorem 4.17]{AFP}), we 
find that, for $\H^{N-1}$-a.e. $x\in \{0\}{\times}(-k,k)^{N-1}$, $x\in \partial^* E^-_i \cap \partial^* E^+_j$ 
for some $i,j$ (where $\partial^* E$ denotes the reduced boundary of $E$). 
Therefore, since $v\in W^{1,\infty}((-1,1)^N;\R^N)$,  Lemma \ref{connessioni} implies that  
there exist rank-1 connections between 
$O(N)H$ and $\lambda\, O(N)H$. This implies in particular that $\lambda=1$,
which is a contradiction to $\lambda\in(0,1)$.
Hence \eqref{lower-b} follows.
\end{proof}
\begin{remark}\label{rmk:nonint}
An estimate similar to \eqref{lower-b} was proven in \cite{LPS,LPS2} 
(for a hexagonal lattice in dimension two and 
a class of three-dimensional lattices) 
via a different argument, based on the non-interpenetration condition. 
In fact, in \cite{LPS,LPS2} a stronger result is proven,
namely, that $\ga(I,\lambda I;k)$ scales like $k^N$.
\par
The non-interpenetration assumption turns out to be necessary if the energy involves only nearest neighbour interactions; 
indeed, in such a case, one can exhibit deformations that violate the non-interpenetration condition and for 
which \eqref{lower-b} does not hold, see \cite[Section 4.2]{LPS}. 
Such deformations, which consist of suitable foldings of the lattice,  would be 
energetically expensive (and, in particular, would not provide a counterexample to \eqref{lower-b}) in the 
present setting, 
exactly because of the effect of the interactions across 
neighbouring cells. It is the latter ones that prevent folding phenomena and 
allow one  
to prove \eqref{lower-b}, via Theorem \ref{thm:compactness}.
\end{remark}
\subsection{Compactness and lower bound}\label{questa}
Before characterising the $\Ga$-convergence for the rescaled functionals \eqref{funzionale-risc}, we show a compactness theorem for sequences with equibounded energy, as well as bounds from above and from below on those functionals in terms of the changes of orientation in the wire. Such bounds will be used in the proof of the $\Ga$-convergence results, Theorems \ref{thm3} and \ref{thm4}.
\par
An essential tool for the compactness and the lower bound is Theorem \ref{thm-rigidity}, which we can apply thanks to the controls provided by Lemma \ref{lemma-equiv}. More precisely, in the part of the wire with $x_1\in(-L,0)$ we use \eqref{rigidity} or its ``symmetric'' version for $O(N){\setminus}SO(N)$ in subdomains that scale in such a way that the constant of the rigidity estimate does not change; for $x_1\in(0,L)$ we use corresponding estimates for $\lambda\,SO(N)$ or $\lambda(O(N){\setminus}SO(N))$. Thus we approximate the deformation gradient with piecewise constant matrices in $O(N)$, respectively $\lambda\,O(N)$.
\par
Due to the fact that a minimum energy has to be paid for each change of orientation, see Lemma \ref{costo-inversione}, the parts with positive determinant do not mix with those with negative determinant. Hence, passing to the weak* limit we obtain functions taking values in $\co(SO(N)) \cup \co(O(N) {\setminus}  SO(N))$, respectively $\lambda\,\co(SO(N)) \cup \lambda\,\co(O(N) {\setminus}  SO(N))$.
Here, $\co(X)$ denotes the convex hull of a set $X$ in $\Mnn$.
\begin{remark}
It is well known that $\co(SO(N)) \cap \co(O(N) {\setminus}  SO(N)) \neq\emptyset$: indeed, the intersection always contains the zero matrix, here denoted by $0$.
In dimension $N=2$, one can see that 
\[
\co(SO(2)) = \left\{
\begin{pmatrix}
\alpha & -\beta \\ 
\beta & \alpha 
\end{pmatrix}
\colon
\alpha^2 + \beta^2 \le 1
\right\} \,, 
\quad
\co(O(2) {\setminus}SO(2)) = \left\{
\begin{pmatrix}
\alpha & \beta \\ 
\beta & -\alpha
\end{pmatrix}
\colon
\alpha^2 + \beta^2 \le 1
\right\} \,.
\]
In particular, $\co(SO(2)) \cap \co(O(2) {\setminus}  SO(2)) = \{0\}$.
For $N>2$, the intersection is nontrivial. For example, $\co(SO(3)) \cap \co(O(3) {\setminus}  SO(3))$ contains the matrix $-\frac13I$.
Moreover, one can see that
\[
 \co(SO(N)) \cup \co(O(N) {\setminus}  SO(N)) \subsetneq \co(O(N))
\]
for $N\ge2$.
\end{remark}
Henceforth, 
the symbol $\U$   stands for the class of subsets of $(-L,L)$ 
that are disjoint union of a finite number of open intervals.
%
%
%
%
\begin{proposition}\label{comp-nano}
Let $\ut_\e\in \widetilde\A_{\eps}(\Om_{k})$ be a 
sequence such that 
\begin{equation} \label{succ-limitata}
\limsup_{\e\to 0^{+}} \I\unolambda_{\e}(\ut_\e,k) \leq C \,.
\end{equation}
Then there exist functions $\ut \in W^{1,\infty}(\Om_k;\R^N)$, $d_1,\dots,d_N\in L^{\infty}(\Om_k;\R^N)$, and a subsequence (not relabelled) such that 
\begin{alignat}{3}
\nonumber\ut_\e - \fint_{\Om_k} \ut_\e \,\d x & \weakst \ut && \quad \text{weakly* in}\ W^{1,\infty}(\Om_k;\R^N)\,,
\\
\label{wstarconv}
\D \ut_\e A_\e^{-1} = (\nabla u_\e) \circ A_\e & \weakst (\partial_1 \ut \, | \, d_2 \, | \cdots \,| d_N) && \quad \text{weakly* in}\ L^{\infty}(\Om_k;\Mnn)\,,
\end{alignat}
and $\ut$, $d_1,\dots,d_N$, 
are independent of $x_2, \dots, x_N$, i.e., 
$\partial_{j}\ut = \partial_{j} d_i = 0$, for each $i=2,\dots,N$ and $j=2,\dots,N$.
Moreover, there exists $U\in\U$ such that 
\begin{equation} \label{charact}
(\partial_1 \ut \, | \, d_2 \, | \cdots \,| d_N)\in 
\begin{cases}
\co(SO(N))H     &  \text{a.e.\ in}\  (-L,0)\cap U \,, \\
\co(O(N) {\setminus} SO(N))H   &  \text{a.e.\ in}\  (-L,0)\setminus U\,,\\
\lambda\,\co(SO(N))H   &  \text{a.e.\ in}\  (0,L) \cap U \,, \\
\lambda\,\co(O(N) {\setminus}  SO(N))H   &  \text{a.e.\ in}\  (0,L)\setminus U \,,
\end{cases}
\end{equation}
and 
\begin{equation} \label{gammaliminf}
\begin{split}
\liminf_{\e\to 0^{+}} \I\unolambda_{\e}(\tilde u_\e,k)\geq & \ 
\gamma(I,J;k) \, \H^0(\partial U\cap(-L,0)) 
+ \gamma(\lambda I,\lambda J;k) \, \H^0(\partial U\cap(0,L)) \\
& + \gamma(I,\lambda I;k) \left[ 1-\chi_{\partial U}(0) \right]  
+ \gamma(I,\lambda J;k)\,\chi_{\partial U}(0)  
\,.
\end{split}
\end{equation}
%
%
%
\end{proposition}
\begin{remark}
The right-hand side of \eqref{gammaliminf} contains different contributions.
The first term corresponds to the minimal energy needed to bridge a rotation with a reflection, or viceversa,
in the left part of the nanowire; the energy spent depends on the number of changes of orientation,
i.e., on the cardinality of $\partial U$.
The second term plays an analogous role for the right part of the nanowire.
The remaining terms describe the interfacial energy spent to bridge the two energy wells $O(N)H$ and $\lambda\, O(N)H$:
that contribution also depends on whether or not the orientation is preserved across the interface,
i.e., on whether $0$ is an inner or external, or boundary  point for $U$.
\end{remark}
\begin{proof}
({\it Compactness})
The assumption \eqref{succ-limitata} implies that $\{\nabla u_\e\}$, respectively $\{\nabla \ut_\e A_\e^{-1}\}$, 
is uniformly  bounded in $L^{\infty}(\Om_{k\eps};\Mnn)$, respectively $L^{\infty}(\Om_k;\Mnn)$. 
(Recall that $u_\e(x) = \tilde{u}_\e(A_\e^{-1} x)$.)
Therefore there exist a subsequence of $\{\ut_\e\}$ (not relabelled)
and functions $\ut\in  W^{1,\infty}(\Om_k;\R^N)$ and $d_i \in L^{\infty}(\Om_k;\R^N) $ for $i=2,\dots, N$, 
such that
$\partial_{1}\ut_\e \weakst \partial_{1} \ut$ weakly* in $L^{\infty}(\Om_k;\Mnn)$, where 
$\ut$ is independent of $x_i$ for all $i=2,\dots, N$,
and $ \frac{1}{\e}\partial_{i}\ut_\e \weakst d_i$ for each $i=2,\dots, N$.

In order to show $\partial_{j}d_i=0$ and \eqref{charact}, we  apply the rigidity estimate \eqref{rigidity} to the sequence $u_\e$.
To this aim,
we divide the 
domain $\overline{\Omega}_{k\e}$ into subdomains that are the Cartesian product of 
intervals $(a_i, a_i +\e)$, $a_i\in\e\Z$, and the cross-section $(-k\e,k\e)^{N-1}$. 
We first observe that, by Lemma \ref{costo-inversione} and assumption \eqref{succ-limitata},
the number of changes of orientation of $u_\e$ is uniformly bounded in $\e$.
More precisely, we can find a uniformly bounded number of subdomains 
$(a_i, a_i +\e)\times(-k\e,k\e)^{N-1}$, $i\in I_\e$, $\# I_\e \le C$,
such that if $i\notin I_\e$ then $\det\nabla u_\e$ has constant sign in $(a_i, a_i +\e)\times(-k\e,k\e)^{N-1}$.
In each of these subdomains, we use \eqref{bound-from-below} to apply the rigidity estimate \eqref{rigidity}, or its ``symmetric'' version for $O(N){\setminus}SO(N)$. 
%
%
\par
Specifically, for each $a_i$ with $a_i<0$ and $i\notin I_\e$,  there exists $P_\e (a_i)\in O(N)H $ such that  
\[
\int_{(a_i,a_i+\e)\times (-k\e,k\e)^{N-1}} 
|\nabla u_\e-P_\e(a_i)|^{p} \,\d x \leq   C\!\!
\int_{(a_i,a_i+\e)\times (-k\e,k\e)^{N-1}} 
\dist^{p}(\nabla u_\e,O(N)H)  \,\d x \,,
\]
and 
for every $a_i >0$ with $i\notin I_\e$ there exists $P_\e (a_i)\in \lambda \, O(N) H $ such that  
\[
\int_{(a_i,a_i+\e)\times (-k\e,k\e)^{N-1}} 
|\nabla u_\e-P_\e(a_i)|^{p} \,\d x \leq   C\!\!
\int_{(a_i,a_i+\e)\times (-k\e,k\e)^{N-1}} 
\dist^{p}(\nabla u_\e, \lambda\,O(N) H)  \,\d x \,.
\]
Moreover for $i\in I_\e$ we set $P_\e(a_i)=I$ if $a_i<0$ and $P_\e(a_i)=\lambda I$ if $a_i\ge0$.
By interpolation one defines a piecewise constant matrix field 
$P_\e:(-L,L)\to O(N)H \cup \lambda \,O(N) H$ such that 
$P_\e(x_{1})=P_\e(a_i)$ if $x_{1}\in (a_i,a_i+\e)$.
Summing up 
over $i$ and rescaling the variables, one gets  
\begin{subequations}\label{bound-compattezza}
\begin{align}
\int_{\Om_k^-} 
|\nabla \ut_\e A_\e^{-1} - P_\e(x_1)|^{p} \,\d x \leq   C\!\!
\int_{A_\e^{-1}(\overline{\Om}_{k\e}) \cap\{x_1<0\}} 
\dist^{p}(\nabla \ut_\e A_\e^{-1},O(N)H )  \,\d x  \leq C\e\,,\\
\int_{\Om_k^+} 
|\nabla \ut_\e A_\e^{-1} - P_\e(x_1)|^{p} \,\d x   \leq   C\!\!
\int_{A_\e^{-1}(\overline{\Om}_{k\e}) \cap\{x_1>0\}} 
\dist^{p}(\nabla \ut_\e A_\e^{-1},\lambda\,O(N) H )  \,\d x \leq C\e\,,
\end{align}
\end{subequations}
where the last inequality of each line follows by applying Lemma \ref{lemma-equiv} to each subdomain with $i\notin I_\e$
and by recalling that each subdomain has volume proportional to $\e$ after rescaling.
\par
%
%
We now define the sets
\begin{align*}
K_\e &:= \{ a_i^\e \in (-L,L) \colon P_\e(x_1) \in SO(N)H \cup \lambda\,SO(N) H \text{ for }
x_1\in [a_i^\e, a_i^\e + \e)\} \,,\\
U_\e &:= \bigcup_{a_i^\e\in K_\e}  [a_i^\e, a_i^\e + \e)\,,
\end{align*}
and remark that Lemma \ref{lemma-equiv},
Lemma \ref{costo-inversione}, 
and assumption \eqref{succ-limitata} imply that the 
cardinality of $\partial U_\e$ is uniformly bounded.
Therefore the sequence $\{\chi_{U_\e}\}$ converges, up to subsequences, to $\chi_{U}$
strongly in $L^1(-L,L)$, where 
\be\label{U}
U = \bigcup_{i=1}^n (\alpha_i,\beta_i) \,, \quad 
-L\le\alpha_1<\beta_1<\alpha_2<\beta_2<\dots<\alpha_n<\beta_n\le L \,.
\ee
%
Since we can write
$$
P_\e(x_1) = R_\e(x_1)\big(\chi_{U_\e\cap (-L,0)} H + \chi_{U_\e\cap (0,L)}\lambda H \big) + 
JR_\e(x_1)\big((1-\chi_{U_\e\cap (-L,0)}) H + (1-\chi_{U_\e\cap (0,L)})\lambda H \big),
$$
where $R_\e : (-L,L) \to SO(N)$ is piecewise constant, we deduce that $P_\e$ converges, 
up to subsequences, to some $P\in L^\infty((-L,L);\Mnn)$ in the weak* topology of 
$L^\infty((-L,L);\Mnn)$. 
From \eqref{bound-compattezza} it follows that the weak* limit of  $\nabla \ut_\e A_\e^{-1}$ 
coincides with $P$ and therefore does not depend on $x_j$ for each $j=2,\dots, N$. Moreover, 
inclusion \eqref{charact} follows from the fact that 
$\chi_{U_\e}P_\e$ converges weakly* to $\chi_{U}P$.
\par\medskip
({\it Lower bound}) 
Inequality \eqref{gammaliminf} is proven by a standard argument which can be found, for example, in 
\cite{LPS,mor-mue,mue-pal}. 
We will briefly sketch the main ideas and refer the reader to \cite{LPS,mor-mue,mue-pal} for full details.
First recall that $\partial U$ 
consists of a finite number of points, cf.\ \eqref{U}.
Since $\chi_{U_\e} \to \chi_U$ and since the number of points of $\partial U_\e$ is uniformly 
bounded, one can find $\sigma >0$, $\alpha_i^\e \to \alpha_i$, $\beta_i^\e \to \beta_i$ such that
\begin{subequations}\label{intervalli} 
\begin{align}
& (\alpha_i^\e-2\sigma , \alpha_i^\e  - \sigma )\subset (-L,L)\setminus U_\e \,, 
&& (\alpha_i^\e+\sigma , \alpha_i^\e  +2 \sigma ) \subset U_\e \,, \\
& (\beta_i^\e-2\sigma , \beta_i^\e  - \sigma ) \subset U_\e \,, 
&& (\beta_i^\e+\sigma , \beta_i^\e  +2 \sigma ) \subset  (-L,L)\setminus U_\e \,.
\end{align}
\end{subequations}
Moreover, if $\sigma$ is sufficiently small, all the 
intervals $(\alpha_i^\e-2\sigma , \alpha_i^\e  + 2 \sigma )$ and 
$ (\beta_i^\e-2\sigma , \beta_i^\e  + 2 \sigma ) $ are mutually disjoint and therefore it 
suffices to prove the lower bound for one of such intervals. 
Suppose that $\alpha_i^\e \in (0,L)$ and 
define 
$$
v_\e(x_1, x_2,\dots,x_N) := \tfrac{1}{\e}\ut_\e(\e x_1 + \alpha_i^\e, x_2,\dots,x_N) = 
\tfrac{1}{\e} u_\e(\e x_1 + \alpha_i^\e,\e x_2,\dots, \e x_N). 
$$
Then, $\nabla v_\e (x)= \nabla \ut_\e (\e x_1 + \alpha_i^\e, x_2,\dots,x_N) A_\e^{-1} 
= 
\nabla u_\e(\e x_1 + \alpha_i^\e,\e x_2,\dots, \e x_N)$, and, by \eqref{bound-compattezza}, we 
have
\be\label{ci-siamo-quasi}
\begin{split}
& \int_{(-\frac{2\sigma}{\e}, -\frac{\sigma}{\e})\times (-k,k)^{N-1} } 
\dist^{p}(\nabla v_\e, \lambda(O(N){\setminus} SO(N)) H  )  \,\d x 
\\
& + \int_{(\frac{\sigma}{\e}, \frac{2\sigma}{\e})\times (-k,k)^{N-1}} 
\dist^{p}((\nabla v_\e,  \lambda\,SO(N)H)  \,\d x 
 \leq C \,.
\end{split}
\ee
From \eqref{ci-siamo-quasi}, Theorem \ref{thm-rigidity} and the Poincar\'e inequality, 
we deduce that there exists a unit interval contained in  
$(-\frac{2\sigma}{\e}, -\frac{\sigma}{\e})$ such that in the Cartesian product of such interval with 
the cross-section $(-k,k)^{N-1}$, 
the $W^{1,p}$-norm of the difference between 
$v_\e$ and an affine map of the form 
$\lambda Q H x + a$, with 
$Q\in O(N){\setminus}SO(N)$ and $a\in \R^N$, is bounded by $C \e/\sigma$. 
By the same argument 
one can find a unit interval contained in $(\frac{\sigma}{\e}, \frac{2\sigma}{\e})$ such that 
in the Cartesian product of such interval with 
the cross-section $(-k,k)^{N-1}$, 
the $W^{1,p}$-norm of the difference between 
$v_\e$ and an affine map of the form 
$\lambda RH x + b$, with 
$R\in SO(N)$ and $b\in \R^N$, is bounded by $C \e/\sigma$.
By gluing the function $v_\e$ with these maps on such intervals, 
one can define a function $\hat{v}_\e\in \A_{\infty}(\Om_{k,\infty})$ 
that is a competitor for  $\gamma (\lambda J,\lambda I;k)$ and 
such that (cf.\ \eqref{24})
$$
\I\unolambda_\e (\tilde u_\e, k)|_{(\alpha_i^\e -2\sigma , \alpha_i^\e + 2\sigma )\times (-k,k)^{N-1}}
\geq   \E\lambdalambda_{\infty}(\hat v_\e,k) - C\frac{\e}{\sigma}\,, 
$$
where 
$\I\unolambda_\e (\tilde u_\e, k)|_{(\alpha_i^\e -2\sigma , \alpha_i^\e + 2\sigma )\times (-k,k)^{N-1}}$ only takes into account 
the interactions between atoms lying in the subset 
$(\alpha_i^\e -2\sigma , \alpha_i^\e + 2\sigma )\times (-k,k)^{N-1}$.
Arguing in a similar way for the other intervals in \eqref{intervalli} yields \eqref{gammaliminf}.
\end{proof}
%
%
%
%
%
\subsection{Upper bound}\label{pure-questa}
We prove that the bound \eqref{gammaliminf} is in fact optimal.
\begin{proposition}\label{upper}
Let $F\in L^{\infty}((-L,L);\Mnn)$ and $U\in\U$  satisfy
\begin{equation} \label{cohull}
F\in
\begin{cases}
\co(SO(N))H &  \text{a.e.\ in}\ (-L,0)\cap U \,,\\ 
\co(O(N){\setminus} SO(N))H &  \text{a.e.\ in}\ (-L,0)\setminus U \,,\\
\lambda\,\co(SO(N))H &  \text{a.e.\ in}\ (0,L)\cap U \,,\\ 
\lambda\,\co(O(N){\setminus} SO(N))H &  \text{a.e.\ in}\ (0,L)\setminus U \,.
\end{cases}
\end{equation}
Then there exists a sequence $\{\ut_\e\}\subset \widetilde\A_{\eps}(\Om_{k})$ such that
\begin{equation} \label{1328061}
\D \ut_\e A_\e^{-1}\weakst F
\quad \text{weakly* in}\ L^{\infty}(\Om_k;\Mnn)\,,
\end{equation}
and
\be\label{1328062}
\begin{split}
\limsup_{\e\to 0^{+}} \I\unolambda_{\e}(\tilde u_\e,k)\leq & \ 
\gamma(I,J;k) \, \H^0(\partial U\cap(-L,0)) 
+ \gamma(\lambda I,\lambda J;k) \, \H^0(\partial U\cap(0,L)) \\
& + \gamma(I,\lambda I;k) \left[ 1-\chi_{\partial U}(0) \right]  
+ \gamma(I,\lambda J;k)\,\chi_{\partial U}(0)  
%
%
\,.
\end{split}
\ee
\end{proposition}
\begin{proof}
Using a standard approximation argument we may assume that $x_1\mapsto F(x_1)$
is piecewise constant, with values in $O(N)H$ for a.e.\ $x_1\in(-L,0)$ and values in $\lambda\, O(N)H$ for a.e.\ $x_1\in(0,L)$.
We may also assume that this approximation process does not modify the set $U$ of \eqref{cohull}.
More precisely, there exist $m,n\in\Z$, $m<0$, $n\ge0$,
$-L=a_m<a_{m+1}<\dots<a_{-1}<a_0 =0 <a_1<\dots<a_n<a_{n+1}=L$,
and $R_{i}\in O(N)$ for $i=m,\dots,-1,0,\dots,n$ such that
\[
F = \sum_{i=m}^{-1}\chi_{(a_i,a_{i+1})}R_{i}H+
 \sum_{i=0}^{n}\chi_{(a_i,a_{i+1})}\lambda R_{i}H
\]
and 
\[
U = \intern \bigcup \big\{ [a_i,a_{i+1}] \colon R_i\in SO(N) \,, \  m\le i\le n-1 \big\} \,.
\]
\par
The following construction is similar to that in \cite[Proposition 3.2]{LPS},
so we will show the details only for what concerns the changes of orientation.
We introduce a mesoscale  $\{\sigma_{\e}\}$ such that $\e \ll\sigma_{\e}\ll 1$ as $\e\to0^+$.
Next we define $\tilde u_{\e}$ in the sets of the type $(a_i{+}\sigma_\e,a_{i+1}{-}\sigma_\e)\times(-k,k)^{N-1}$ in such a way that 
its gradient equals $R_i H A_\e$ if $a_{i+1}\le0$ and equals $\lambda R_i H A_\e$ if $a_i\ge 0$.
This determines $\tilde u_{\e}$ in those regions, up to some additive constants that will have to be fixed 
at the end of the construction in order to make $\tilde u_{\e}$ continuous.
\par
We now complete the definition of $\tilde u_{\e}$ in the sets of the type $(a_i{-}\sigma_\e,a_{i}{+}\sigma_\e)\times(-k,k)^{N-1}$.
Let us first assume $i<0$, i.e., $a_i<0$. Since $R_{i-1}$ and $R_{i}$ may be in $SO(N)$ or in $O(N){\setminus} SO(N)$, one can have four cases.
If both $R_{i-1}$ and $R_i$ are in $SO(N)$, it is possible to define $\tilde u_{\e}$ 
by interpolating $R_{i-1}$ and $R_i$ so that
the cost of the transition has  order $O(\e/\sigma_\e)$, so it gives no contribution to \eqref{1328062};
we refer to \cite{LPS} for details.
The case $R_{i-1},R_i\in O(N){\setminus} SO(N)$ is completely analogous.
\par
If $R_{i-1}\in SO(N)$ and $R_i\in O(N){\setminus} SO(N)$ or viceversa, we define $\tilde u_{\e}$ in $(a_i{-}\sigma_\e,a_{i}{+}\sigma_\e)\times(-k,k)^{N-1}$
as a rescaling of a quasiminimiser of \eqref{gamma2}. More precisely, we fix $\eta>0$ and apply the definition of $\ga(R_{i-1},R_i;k)$,
thus finding $M>0$ and $v\in \A_{\infty}(\Om_{k,\infty})$ such that 
\[
 \D v=R_{i-1}H  \ \text{for} \ x_1\in(-\infty,-M)\,,\, \quad \D v=R_{i}H \ \text{for} \ x_1\in(M,+\infty)
\]
and
\[
\E\unouno_{\infty}(v,k)  \leq \ga(I,J;k) + \eta \,,
\]
where we used also Proposition \ref{invariance}.
With this at hand, we define $\tilde u_{\e}$ in $(a_i{-}\sigma_\e,a_{i}{+}\sigma_\e)\times(-k,k)^{N-1}$ as
\[
\tilde u_{\e}(x):= \eps v(\tfrac{1}{\eps}A_\eps x) + b \,.
\]
The constant vector $b$ in the last equation is chosen in such a way that $\tilde u_{\e}$ is continuous.
Since each point of $\partial U$ gives the same contribution $\ga(I,J;k)$ to the upper bound,
we obtain the first term of \eqref{1328062}.
\par
The case $i>0$, i.e., $a_i>0$, is treated similarly to $i<0$ and gives rise to the second term of  \eqref{1328062}.
Finally, for $i=0$, i.e., $a_i=a_0=0$, we argue as above and define $\tilde u_{\e}$ by using a rescaling of a quasiminimiser of \eqref{gamma1} 
and applying the definition of $\ga(R_{-1},\lambda R_0;k)$. 
We then get an interfacial contribution in \eqref{1328062} that differs in the two cases
$0\in \partial U$ and $0\notin\partial U$.
\end{proof}
\subsection{Limit functionals with respect to different topologies}\label{spero-ultima}
In the next theorem we characterise the $\Gamma$-limit of the sequence $\{  \I\unolambda_\e(\cdot,k) \}$ 
with respect to the weak* convergence in $W^{1,\infty}(\Om_k;\R^N)$.
As it can be inferred from the compactness result in Proposition \ref{comp-nano},
the domain of the $\Gamma$-limit turns out to be
\begin{equation}
\label{gammadomain}
\begin{split}
\A\unolambda(k):=\big\{
u\in W^{1,\infty}(\Om_k;\R^N)\colon &
\partial_{2} u  = \dots  =  \partial_{N} u  = 0 \ \text{a.e.\ in}\ \Om_k\,, \\
& |\partial_1 u |\leq 1 \ \text{a.e.\ in}\ \Om_k^- \,, \ 
|\partial_1 u |\leq \lambda \ \text{a.e.\ in}\ \Om_k^+
\big\} \,.
\end{split}
\end{equation}
We show that on such domain the $\Gamma$-limit is constant.
Hence, the macroscopic description of the model is similar to that of \cite{LPS,LPS2};
in particular, it does not have memory of the changes of orientation in minimising sequences.
In order to keep track of the orientation changes, we need to introduce a stronger topology for the $\Gamma$-convergence,
as we see in Theorem \ref{thm4}.
\par
\begin{theorem}\label{thm3}
The sequence of functionals $\{  \I\unolambda_{\e}(\cdot,k)\}$ 
$\Gamma$-converges, as $\e\to 0^{+}$, to the functional
\begin{equation}\label{eqthm3}
\I\unolambda(u,k)=
\begin{cases}
\ga(k)   &  \text{if}\ u\in\A\unolambda(k)\,,\\
+\infty   &  \text{otherwise,}
\end{cases}
\end{equation}
with respect to the weak* convergence in $W^{1,\infty}(\Om_k;\R^N)$, where
\be\label{gammamin}
\gamma(k):= \min\big\{ \gamma(I,\lambda I;k) , \gamma(I,\lambda J;k) \big\} \,.
\ee
\end{theorem}
\begin{proof}
({\it Liminf inequality})
Let $ \ut_\e \in \widetilde\A_{\eps}(\Om_{k})$ be a sequence of functions converging to a function $u$ weakly* in $W^{1,\infty}(\Om_k;\R^N)$. 
We have to show that 
\bes
\I\unolambda(u,k)\leq  \liminf_{\e \to 0^+} \I\unolambda_\e(\ut_\e,k) \,.
\ees
We assume that $\liminf_{\e \to 0^+} \I\unolambda_\e(\ut_\e,k)  \le  C$, the other case being trivial. 
By applying Proposition \ref{comp-nano} we find a set $U\in\U$ and functions $\ut \in W^{1,\infty}(\Om_k;\R^N)$, 
$d_2,\dots,d_N\in L^{\infty}(\Om_k;\R^N)$  independent of $x_2,\dots,x_N$, 
such that \eqref{wstarconv}, \eqref{charact}, and \eqref{gammaliminf} hold.
This implies that $\partial_1u=\partial_1\ut$ a.e., the function $u$ is independent of $x_2,\dots,x_N$, and $u\in \A\unolambda(k)$.
Notice that the right-hand side of \eqref{gammaliminf} is greater than or equal to $\gamma(k)$, since
$\gamma(\cdot,\cdot;k)$ is positive.
\par\medskip
({\it Limsup inequality})
Given a function $u\in W^{1,\infty}(\Om_k;\R^N)$ we have to find a sequence 
$\{\ut_\e\}\subset \widetilde\A_{\eps}(\Om_{k})$ such that
$\ut_\e \weakst u$ weakly* in $W^{1,\infty}(\Om_k;\R^N)$ and 
\begin{equation}\label{ineq2}
\limsup_{\e\to 0^+} \I\unolambda_\e (\ut_\e,k)\leq \I\unolambda(u,k) \,.
\end{equation}
We assume that $u\in\A\unolambda(k)$, the other case being trivial.
\par
The construction of the recovery sequence depends on the precise value of the minimum in \eqref{gammamin}.
Since we do not know such value, we explain how to proceed in
the case when $\gamma(k)$ is any of the two quantities therein.
\begin{itemize}
\item If $\gamma(k)= \gamma(I,\lambda I;k)$, we set $U:=(-L,L)$ and,
following e.g.\ \cite[Theorem 4.1]{mor-mue}, we construct measurable functions
$d_{2},\dots,d_N\in L^{\infty}(\Om_k;\R^{N})$, independent of $x_2,\dots,x_N$, such that 
\begin{equation*}
(\partial_1 u \, | \, d_2 \, | \, \cdots \, | \, d_N)\in 
\begin{cases}
\co(SO(N))H     &  \text{a.e.\ in}\ \Om_k^-\,,\\
\lambda\,\co(SO(N))H   &  \text{a.e.\ in}\ \Om_k^+\,. 
\end{cases}
\end{equation*}
\item If $\gamma(k)=\gamma(I,\lambda J;k)$ we set $U:=(-L,0)$ and construct $d_{2},\dots,d_N$ in such a way that
\[
(\partial_1 u \, | \, d_2 \, | \, \cdots \, | \, d_N)\in 
\begin{cases}
\co(SO(N))H     &  \text{a.e.\ in}\ \Om_k^-\,,\\
\lambda\,\co(O(N){\setminus}SO(N))H   &  \text{a.e.\ in}\ \Om_k^+\,.
\end{cases} 
\]
\end{itemize}
Proposition \ref{upper} can be now applied to $F:=(\partial_1 u \, | \, d_2 \, | \, \cdots \, | \, d_N)$,
hence providing us with a sequence 
$\{\ut_\e\}\subset  \widetilde\A_{\eps}(\Om_{k})$ satisfying  \eqref{1328061}--\eqref{1328062}.
In particular we have $\nabla\ut_\eps\weakst\nabla u$ weakly* in $L^{\infty}(\Om_k;\Mnn)$ and
\eqref{ineq2} holds because of the choice of $U$ and the definition of $\gamma(k)$.
\end{proof}
\begin{figure} 
\centering
\includegraphics[width=.99\textwidth]{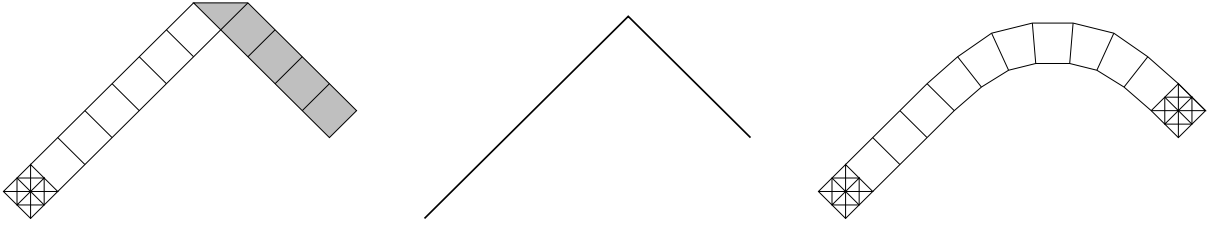}
\caption{Two possible recovery sequences for the profile at the centre of the figure.
Here we picture only a part of the wire containing just one species of atoms,
therefore the transition at the interface is not represented.
A kink in the profile may be reconstructed by folding the strip, i.e., mixing rotations and reflections (left);
or by a gradual transition involving only rotations or only reflections (right).
In the limit, the former recovery sequence gives a positive cost, while the latter gives no contribution.
If the stronger topology is chosen, the appropriate recovery sequence will depend on the value of the internal variable,
which defines the orientation of the wire.}
\label{fig:transitions}
\end{figure}
\begin{remark}\label{rem100}
As long as the $\Gamma$-convergence is taken with respect to the weak* topology of $W^{1,\infty}(\Om_k;\R^N)$,
\eqref{eqthm3} only accounts for the cost of transitions at the interface
between the two species of atoms.
Indeed, away from the interface it is always possible to construct recovery sequences
without mixing rotations and reflections, as done in the proof of the limsup inequality;
such transitions have low interaction energy, since $\gamma(I,I)=\gamma(J,J)=0$,
see also Proposition \ref{invariance}.
In particular, for $\lambda=1$ the limit functional is trivial, since $\I\unouno(u,k)=0$ if $u\in\A\unouno(k)$.
\par
Below we show that, if a stronger topology is chosen, the value of the $\Gamma$-limit changes.
The resulting limit functional depends on an internal variable, $D$ in \eqref{eqthm4},
that keeps track of the changes of orientation throughout the thin wire.
In fact, different transitions between the energy wells must now be employed according to the value of $D$;
two examples are provided in Figure \ref{fig:transitions}.
\end{remark}
%
%
%
%
We introduce the functionals defined for $u\in W^{1,\infty}(\Om_k;\R^N)$ and $D\in L^\infty(\Om_k;\Mnn)$
by
\begin{equation*} 
\hat\I\unolambda_{\e}(\tilde u,D,k) := 
\begin{cases}
\I\unolambda_\e(\ut,k) & \text{if}\ \ut \in \widetilde\A_{\e}(\Om_{k}) \ \text{and}\ D= 
(\partial_1 \ut \, | \, \e^{-1}\partial_2 \ut \, | \, \cdots \, | \, \e^{-1}\partial_N \ut)\,, \\
+\infty & \text{otherwise.}
\end{cases}
\end{equation*}
%
%
In the next theorem we study the $\Gamma$-limit of the sequence $\{  \hat\I\unolambda_\e(\cdot,\cdot,k) \}$ 
as $\e\to0^+$
with respect to the weak* convergence in $W^{1,\infty}(\Om_k;\R^N)\times L^\infty(\Om_k;\Mnn)$.
As a consequence of Proposition \ref{comp-nano}, the domain of the $\Gamma$-limit turns out to be
\bes
\begin{split}
\hat\A\unolambda(k):=\big\{
(u,D) \colon & u{\in}\A\unolambda(k) \,, \:  D{\in} L^\infty(\Om_k;\Mnn) \,, \: De_1{=}\partial_{1} u \,, \:
\partial_{2} D  {=} \dots  {=}  \partial_{N} D  {=} 0 \ \text{a.e.\ in}\ \Om_k\,, \\
& 
D \in 
\co(SO(N))H \cup \co(O(N){\setminus}SO(N))H \ \text{a.e.\ in}\ \Om_k^-\,,\\
&
D \in 
\lambda\,\co(SO(N))H \cup \lambda\,\co(O(N){\setminus}SO(N))H \  \text{a.e.\ in}\ \Om_k^+
\big\} \,,
\end{split}
\ees
where $\A\unolambda(k)$ is defined by \eqref{gammadomain}.
It is convenient to introduce the following definition, where
the functional $\J$ coincides with the right-hand sides of \eqref{gammaliminf} and \eqref{1328062}.
\begin{definition}
Given $(u,D)\in\hat\A\unolambda(k)$, let $\U(u,D)$ be the collection of all subsets $U\in\U$ such that
\be\label{coso}
D \in 
\begin{cases}
\co(SO(N))H &  \text{for a.e.}\ x_1\in(-L,0)\cap U \,,\\ 
\co(O(N){\setminus} SO(N))H &  \text{for a.e.}\ x_1\in(-L,0)\setminus U \,,\\
\lambda\,\co(SO(N))H &  \text{for a.e.}\ x_1\in(0,L)\cap U \,,\\ 
\lambda\,\co(O(N){\setminus} SO(N))H &  \text{for a.e.}\ x_1\in(0,L)\setminus U \,.
\end{cases}
\ee
For $U\in\U(u,D)$ we set
\[
\begin{split}
\J(U) := & \ 
\gamma(I,J;k) \, \H^0(\partial U\cap(-L,0)) 
+ \gamma(\lambda I,\lambda J;k) \, \H^0(\partial U\cap(0,L)) \\
& + \gamma(I,\lambda I;k) \left[ 1-\chi_{\partial U}(0) \right]  
+ \gamma(I,\lambda J;k)\,\chi_{\partial U}(0)  
%
%
\end{split}
\]
and 
\be\label{Jmin}
\J_\textnormal{min}(u,D) := \min_{U\in \U(u,D)} \J(U) \,.
\ee
\end{definition}
The last definition will be used to apply Propositions \ref{comp-nano} and \ref{upper}
towards the characterisation of the $\Ga$-limit with respect to the stronger topology.
To this end, each pair $(u,D)\in\hat\A\unolambda(k)$ is associated with a set $U$ realising \eqref{coso}.
Such $U$ is in general not unique, since $\co(SO(N)) \cap \co(O(N){\setminus} SO(N)) \neq \emptyset$.
Therefore, we choose it to be ``optimal'', i.e., minimising \eqref{Jmin}.
Notice that the minimum in \eqref{Jmin} is attained since 
\[
\{ \J(U) \colon U{\in}\,\U(u,D) \}
\subset
\{ m_1 \gamma(I,J;k) + m_2 \gamma(I,\lambda I;k)+ m_3 \gamma(I,\lambda J;k)+ m_4 \gamma(\lambda I,\lambda J;k) \colon m_i{\in} \N \} \,.
\]
A minimiser needs not be unique as shown in the following example. 
\begin{example}
Fix $a_1<a_2<0$ and assume that $D(x_1)\in (O(N){\setminus}SO(N))H$ for $x_1<a_1$, $D(x_1)=0$ for $a_1<x_1<a_2$, $D(x_1)\in SO(N)H$ for $a_2<x_1<0$,
and $D(x_1)\in \lambda\,SO(N)H$ for $x_1>0$. Then any interval of the type $U=(a,+\infty)$, with $a_1\le a\le a_2$,
is a minimiser of \eqref{Jmin}.
\end{example}
\begin{theorem}\label{thm4}
The sequence of functionals $\{  \hat\I\unolambda_{\e}(\cdot,\cdot,k)\}$ 
$\Gamma$-converges, as $\e\to 0^{+}$, to the functional
\begin{equation}\label{eqthm4}
\hat\I\unolambda(u,D,k):=
\begin{cases}
\J_\textnormal{min}(u,D)   &  \text{if}\ (u,D)\in\hat\A\unolambda(k)\,,\\
+\infty   &  \text{otherwise,}
\end{cases}
\end{equation}
with respect to the weak* convergence in $W^{1,\infty}(\Om_k;\R^N)\times L^\infty(\Om_k;\Mnn)$,
where $\J_\textnormal{min}(u,D)$ is defined by \eqref{Jmin}.
\end{theorem}
\begin{proof}
The liminf inequality is obtained by applying Proposition \ref{comp-nano} and arguing as in Theorem \ref{thm3}.
Also the derivation of the limsup inequality is similar to the one performed in Theorem \ref{thm3}; 
let us simply point out that, while in the proof of Theorem \ref{thm3} the matrix field $F$ needed to be reconstructed,
here we set $F:=D$ and choose $U$ as a minimiser of \eqref{Jmin}.
The conclusion follows by applying Proposition \ref{upper}.
\end{proof}
\begin{remark}\label{rem-nontrivial}
{\rm
We underline that Theorem \ref{thm4} provides a nontrivial $\Ga$-limit also in the case when $\lambda=1$.
Indeed, one has $\hat\I\unouno(u,D,k)=\ga(I,J;k) \, \H^0(\partial U\cap(-L,L))$ if $(u,D)\in\hat\A\unouno(k)$ and $U$ miminises \eqref{Jmin}, where $\ga(I,J;k)>0$.
}
\end{remark}
%
%
\subsection{Boundary conditions and external forces}\label{sezione-forze}

In the present section we discuss how the previous results extend to the case when the functional \eqref{eng-eps} is complemented by boundary 
conditions or external forces.
Although our considerations apply to the case of general $H\in GL^+(N)$ and $\lambda \in (0,1]$, for simplicity we will focus on the case 
$H=I$ and $\lambda =1$. 
We will also test the consistency of the present model with the non-interpenetration condition by looking at minimisers of 
the $\Gamma$-limit 
when boundary conditions or forces are prescribed. 
We will see that the continuum limit that keeps track of such constraints is the one provided by the stronger topology \eqref{eqthm4}.
\\
\paragraph{\bf Boundary conditions}
Let $B^-,B^+\in GL^+(N)$ and suppose that the functional \eqref{eng-eps} is now defined on deformations $u\in\A_{\eps}(\Om_{k\eps})$ that satisfy
\begin{equation}\label{bound-cond}
\begin{cases}
\nabla u(x)=B^-x & \text{ if }-L < x_1 < -L +\e\,, \\
\nabla u(x)=B^+x & \text{ if }L -\e < x_1 < L \,.
\end{cases}
\end{equation}
It is easy to see that while the compactness result of Proposition \ref{comp-nano} remains valid, the $\Gamma$-limit  \eqref{eqthm3} 
will now contain  
additional terms corresponding to the minimal energy spent to fix the atoms in the vicinity of the lateral boundaries. 
However, such extra terms do not depend on the limiting deformations, therefore they do not 
encode any information about the 
behaviour of minimising sequences.
As far as the stronger topology is concerned, 
one can see that the limit functional 
\eqref{eqthm4} will contain the additional  quantities $\ga(B^-,P;k)$ and $\ga(P,B^+;k)$ defined, for $P\in\{I,J\}$, 
by
\begin{equation}\label{extra1}
\begin{split}
\gamma(B^-,P;k):=\inf\big\{ & \E\unouno_{M}(v,k) \colon M>0\,,
\  v\in \A_{\infty}(\Om_{k,\infty})\,, \\
& \D v=B^- \ \text{for} \ x_1\in(-\infty,-M)\,,\, \ \D v= P \ \text{for} \ x_1\in(M,+\infty) 
\big\}\,,
\end{split}
\end{equation}
\begin{equation}\label{extra2}
\begin{split}
\gamma(P,B^+;k):=\inf\big\{ & \E\unouno_{M}(v,k) \colon M>0\,,
\  v\in \A_{\infty}(\Om_{k,\infty})\,, \\
& \D v=P \ \text{for} \ x_1\in(-\infty,-M)\,,\, \ \D v= B^+ \ \text{for} \ x_1\in(M,+\infty) 
\big\}\,,
\end{split}
\end{equation}
where $\E\unouno_{M}$ is as in \eqref{gamma2}, except that the sum is taken over all atoms contained in the bounded 
strip $(-M,M)\times (-k,k)^{N-1}$.
The choice of $P=I$ or $P=J$
depends on whether or not $\pm L \in \partial \bar U$, where $\bar U$ is a 
minimiser of \eqref{Jmin}. 
Precisely, if $-L \in \partial \bar U$ (resp. $L \in \partial \bar U$), then in \eqref{extra1} (resp. \eqref{extra2}) we take $P=I$, otherwise we 
take $P=J$.  
\begin{remark}\label{ultimissimo}
By Proposition \ref{comp-nano} and the properties of $\Gamma$-convergence, minimisers of \eqref{eng-eps} subjected to 
\eqref{bound-cond} converge, up to subsequences, to minimisers of \eqref{eqthm4} complemented with the above extra terms.
Moreover, if $\dist(B^{\pm};SO(N))$ is sufficiently small, then such minimisers will not have transitions between  $\co\big(SO(N)\big)$ and 
$\co\big(O(N){\setminus} SO(N)\big)$. 
This follows from the fact that $\gamma(I,B^{\pm};k)\to 0$ as $\dist(B^{\pm};SO(N))\to 0$ and therefore, as long as 
$\gamma(I,B^{+};k)+\gamma(B^{-},I;k)< \gamma(I,J;k)$, the optimal transitions  
will fulfil the non-interpenetration condition. 
In this respect the quantity $\gamma(I,J;k)$ can be regarded as an energetic barrier that must be 
overcome in order to have folding effects.
\end{remark}
\paragraph{\bf External forces} 
We study a class of tangential/radial forces acting along the rod.
Let $F_1\in \R^N$, $F_2,\dots,F_{N}\in C^0(\R^N;\R^N)$ be a collection of vector fields such that $F_i=F_i(x_1)$ for every $2=1,\dots,N$. 
We denote by $F$ the matrix field whose columns are $F_1,\dots,F_N$.
For each $u \colon\L_{\e}(k)\to\R^N$, consider the functional 
\be\label{forze}
\begin{split}
\F_\e(u,k):=& \!\!\!\!\! \sum_{(\pm L_\e,x_2,\dots,x_N)\in \L_{\e}(k)}\!\!\!\! \!\!\!\!\!\!\!
 F_1
{\cdot}
\big(
u(L_\e, x_2,\dots,x_N) - 
u(- L_\e, x_2,\dots,x_N)
\big)  \\
&+\!\!\!\!\!\sum_{(x_1,\pm \e k,\dots,x_N)\in \L_{\e}(k)}\!\!\!\!\!\!\!\!\!\!\! F_2(x_1)
\cdot
\big(
u(x_1,\e k,x_3,\dots,x_N) - u(x_1,-\e k,x_3,\dots,x_N)
\big)+ \cdots \\
&\dots + \!\!\!\!\!\sum_{(x_1,\dots,x_{N-1},\pm \e k)\in \L_{\e}(k)}\!\!\!\!\!\!\!\!\!\!\! 
F_N(x_1)
\cdot\big(
u(x_1,\dots,x_{N-1}, \e k) - u(x_1,\dots,x_{N-1}, -\e k)
\big) \,,
\end{split}
\ee
where $L_\e:= L$ if $L$ is an integer multiple of $\e$, and $L_\e:= ([L/\e] +1)\e $ otherwise.
The functional $\F_\e$ consists of several terms: the first sum represents a tangential force, while 
the other terms define a radial force acting on the external atoms of the lattice and enforcing the average displacements along 
the coordinate directions $e_2,\dots,e_N$ to 
be aligned with the given vector fields  $F_2,\dots,F_{N}$. Note that $ \F_\e(u,k)$ can be written as
\[
\begin{split}
\F_\e(u,k)=& \sum_{x_1=-L_\e}^{L_\e-\e} \ \sum_{(x_1,\dots,x_N)\in \L_{\e}(k)}\!\!\!\! \!\!\!\!\!\!\!
 F_1
{\cdot}
\big(
u(x_1+\e,\dots,x_N) - 
u(x_1,\dots,x_N)
\big)  \\
&+ \sum_{x_2=-\e k}^{\e(k-1)} \ \sum_{(x_1,\dots,x_N)\in \L_{\e}(k)}\!\!\!\!\!\!\!\!\!\!\! F_2(x_1)
\cdot
\big(
u(x_1,x_2+\e,\dots,x_N) - 
u(x_1,x_2,\dots,x_N)
\big)+ \cdots \\
&\dots + \sum_{x_N=-\e k}^{\e(k-1)} \ \sum_{(x_1,\dots,x_N)\in \L_{\e}(k)}\!\!\!\!\!\!\!\!\!\!\! F_N(x_1)
\cdot
\big(
u(x_1,\dots,x_N+\e) - 
u(x_1,\dots,x_N)
\big) \,,
\end{split}
\]
hence we have that $\F_\e(u,k)\simeq\frac{1}{\e^{N-1}}\integ{\Om_{k\e}}{F:\nabla u}{x}$.

Introducing the new variables $z(x):=A_\e x$ defined by \eqref{cambio-var}, and 
adopting the notation used in Section  \ref{sec:not},
\eqref{forze} can be equivalently expressed in terms of $\ut(x) := u(z(x))$, namely
\[
\begin{split}
\Ft_\e(\ut,k):= &
\!\!\!\!\!\!\!  \sum_{(\pm L_\e,x_2,\dots,x_N)\in A_\e^{-1}\L_{\e}(k)} \!\!\!\!\!\!\!\!\!\!\!
F_1
{\cdot}
\big(
\ut(L_\e,x_2,\dots,x_N) - \ut(- L_\e, x_2,\dots,x_N)
\big)  \\
&+ \!\!\!\!\!\!\! \sum_{(x_1,\pm k,\dots,x_N)\in A_\e^{-1}\L_{\e}(k)} \!\!\!\!\!\!\!\!\!\!\!
F_2(x_1)
{\cdot}
\big(
\ut(x_1, k,x_3,\dots,x_N) - \ut(x_1,- k, x_3,\dots,x_N)
\big)+ \dots  \\
&\dots +  \!\!\!\!\!\!\! \sum_{(x_1,\dots,x_{N-1},\pm  k)\in A_\e^{-1}\L_{\e}(k)}\!\!\!\!\! \!\!\!\!\!\!
F_N(x_1)
{\cdot}
\big(
\ut(x_1,\dots,x_{N-1},  k) - \ut(x_1,\dots,x_{N-1}, - k)
\big) = \F_\e(u,k)\,,
\end{split}
\]
so that $\Ft_\e(\ut,k)\simeq \integ{\Om_{k}}{F:(\nabla \ut\, A_\e^{-1})}{x}$.
We can then address the study of the asymptotic behaviour of the sequence 
%
%
\begin{equation}\label{nuovo-funzionale}
\G_\e(\tilde u, D,k):=  \hat\I\unouno_{\e}(\tilde u,D,k) - \Ft_\e(\ut,k)\,, \quad \ut\in\widetilde\A_{\e}(\Om_{k}),\,D\in L^\infty(\Om_k;\Mnn) \,.
\end{equation}
Note that in this context we cannot use the weak* convergence in $W^{1,\infty}(\Om_k;\R^N)$, since this 
does not control $\Ft_\e(\ut)$, which is in fact a term depending on $D=\nabla\ut\,A_\e^{-1}$. This justifies the 
choice of $\hat\I\unouno_{\e}$ rather than $\I\unouno_{\e}$ in the definition of $\G_\e$: in order to control 
both terms in the right hand side of \eqref{nuovo-funzionale}, we use the stronger topology provided by  
the weak* convergence in $W^{1,\infty}(\Om_k;\R^N)\times L^\infty(\Om_k;\Mnn)$.
The force term is indeed a continuous perturbation of  $\hat\I\unouno_{\e}$ with respect to such topology.
We observe that
\[
\frac{C}{\e} \int_{\Om_k} ( |\nabla \ut\,A_\e^{-1}|^p -1 ) \, \d x \leq 
\frac{C}{\e}\int_{{\Om}_{k}}\dist^{p}(\nabla \ut\,A_\e^{-1} ,O(N))  \,\d x \leq  \hat\I\unouno_\e(\ut,k) 
\]
and
\[
 \widetilde\F_\e(\ut,k) \leq
 C \Big( \int_{\Om_k} |F|^{p'} \,\d x + \int_{\Om_k} |\nabla \ut\,A_\e^{-1}|^{p} \,\d x \Big) \,.
 \]
Let now $\{(\ut_\e, D_\e)\} \in \widetilde\A_{\e}(\Om_{k}) \times L^\infty(\Om_k;\Mnn)$ be a sequence such that 
$$
\limsup_{\e\to 0^+} \G_\e(\tilde u_\e, D_\e,k) \leq C\,.
$$
The previous inequalities imply that $|| \nabla \ut\,A_\e^{-1} ||_{L^p(\Om_k;\Mnn)}$ is equibounded, which in turn implies that
$\limsup_{\e\to 0^{+}} \hat\I\unouno_{\e}(\ut_\e,k) \leq C$ and thus the conclusions of Proposition \ref{comp-nano} 
are still valid. (See also \cite[Remark 4.2]{mor-mue} for similar results.) 
Taking also into account Theorem \ref{thm4}, we derive the following result.
\begin{theorem}
The following results hold:
\par
\emph{(Compactness) }
Let $\{(\ut_\e, D_\e)\} \in \widetilde\A_{\e}(\Om_{k}) \times L^\infty(\Om_k;\Mnn)$ be a sequence such that 
$$
\limsup_{\e\to 0^+} \G_\e(\tilde u_\e, D_\e,k) \leq C\,.
$$
Then there exists $(\ut,D)\in \hat\A^{1,1}(k)$ and a subsequence (not relabelled) such that 
\begin{alignat*}{3}
\ut_\e - \fint_{\Om_k} \ut_\e \,\d x & \weakst \ut && \quad \text{weakly* in}\ W^{1,\infty}(\Om_k;\R^N)\,,
\\
\D \ut_\e A_\e^{-1} = (\nabla u_\e) \circ A_\e & \weakst D&& \quad \text{weakly* in}\ L^{\infty}(\Om_k;\Mnn)\,.
\end{alignat*}
%
\par
\emph{($\Gamma$-limit)}
 The sequence of functionals $\{ \G_\e \}$ 
$\Gamma$-converges, as $\e\to 0^{+}$, to the functional
\begin{equation}\label{gamma-lim-forze}
\G(u,D,k) :=
\hat\I\unouno(u,D,k) -\widetilde \F(D,k) \,,
\end{equation}
with respect to the weak* convergence in $W^{1,\infty}(\Om_k;\R^N)\times L^\infty(\Om_k;\Mnn)$, where
$$
\widetilde\F(D,k) := (2k)^{N-1}\int_{-L}^L \big(F_1 \cdot d_1 + \cdots + F_{N}\cdot d_N\big) \,\d x_1 
$$
 for $D{=}(d_1| \cdots | d_N)$.
\end{theorem}
As a consequence of the previous theorem and the standard properties of $\Gamma$-convergence we infer the following result about convergence of minima and minimisers.
\begin{corollary}
We have that
$$
\lim_{\e\to 0}\min\{\G_\e (u, D) \colon (u, D)\in  \widetilde\A_{\e}(\Om_{k}) \times L^\infty(\Om_k;\Mnn) \}=\min\{\G(u,D,k) \colon (u,D)\in \hat\A^{1,1}(k)\} \,.
$$
Moreover if $(u_\e,D_\e)\in  \widetilde\A_{\e}(\Om_{k}) \times L^\infty(\Om_k;\Mnn)$ is such that
$$
\lim_{\e\to 0} \G_\e (u_\e, D_\e)=\lim_{\e\to 0} \min\{\G_\e (u, D) \colon (u, D)\in  \widetilde\A_{\e}(\Om_{k}) \times L^\infty(\Om_k;\Mnn) \} \,,
$$
then any cluster point $(\overline u, \overline D)$ of $(u_\e,D_\e)$ with respect to the weak* convergence in $W^{1,\infty}(\Om_k;\R^N)\times L^\infty(\Om_k;\Mnn)$ is a minimiser for $\min\{\G(u,D,k):\, (u,D)\in \hat\A^{1,1}(k)\}$.
\end{corollary}
We now come back to the question of the consistency of the model 
with the non-in\-ter\-pen\-e\-tra\-tion condition.
In this context we cannot expect that minimisers of \eqref{nuovo-funzionale} preserve orientation for the 
whole class of loads defined above. 
This is clarified in the following remark.
\begin{remark}\label{rem-forze}
Minimisers of the functional defined by \eqref{gamma-lim-forze} may have transition points 
between the two wells $SO(N)$ and $O(N){\setminus}SO(N)$.
Suppose for instance that $F_1,\dots,F_N$ satisfy the following properties:  
there exist $n_1,\dots,n_N \in S^{N-1}$, $a\in (-L,L)$, such that 
$(n_1| \cdots | n_N)\in SO(N)$,
$F_i(x_1) = f_{i}(x_1)n_i$ for each $i=1,\dots,N$, $f_1\in\R$, 
$f_{i} > 0$ in $(-L,L)$ for each $i=1,\dots,N-1$,
$f_{N} > 0$ in $(-L,a)$, 
$f_{N} < 0$ in $(a,L)$.

Define $\overline D: =(n_1| \cdots | n_N)$ if $x_1\in(-L,a)$, and 
$\overline D: = (n_1| \cdots| n_{N-1} | -n_N)$ if $x_1\in(a,L)$. 
Note that $(x_1 n_1, \overline D) \in \hat\A\unouno(k)$, 
and   
$\overline D$ has a transition point at $x_1 =a$.  
Denote by $\hat\A\unouno_0(k)$ the subset of $\hat\A\unouno(k)$ of 
deformations with no transitions; i.e., 
$\hat\A\unouno_0(k):= \{(u,D)\in \hat\A\unouno(k) \colon \hat\I\unouno(u,D,k) = 0  \}$.
It is easy to see that 
$$
C:=\!\!\! \min_{(u,D)\in \hat\A\unouno_0(k)} \!\!\! -\widetilde F(D,k) > -\widetilde F(\overline D,k) = 
- (2k)^{N-1}
\Bigg(
 \sum_{i=1}^{N-1}\int_{-L}^L f_i \,\d x_1  
+\int_{-L}^a f_N\, \d x_1- \int_{a}^L f_N \, \d x_1
\Bigg) \,.
$$
Therefore, if $f_1,\dots, f_N$ are such that
$$
-\widetilde F(\overline D,k) +  \gamma(I,J;k) < C,
$$
then it is energetically preferred to have a transition at $a$, namely, all minimisers of $\G$ are 
given by $(x_1 n_1 + b, \overline D)$, with $b$ any vector in $\R^N$.
In contrast, if $f_N$ is always positive, then minimisers will not display any transition.
\end{remark}
%
\subsection{Comparison with models including dislocations}\label{sec:disl}
The lattice mismatch in hetero\-struc\-tured materials, corresponding to $\lambda\neq 1$ in the model described in this section,
can be relieved by creation of dislocations; i.e., line defects of the crystal structure.  
We refer to \cite{egcs,Kavanagh,Schmidtetal2010} for an account of the literature on dislocations in nanowires.
A model for discrete heterostructured nanowires accounting for dislocations was studied in \cite{LPS,LPS2}
under the assumption that deformations fulfil the non-interpenetration condition.
In this paper we have chosen to consider only defect-free configurations in order to both
 simplify the exposition and to pose emphasis on the difficulties to overcome
when the non-interpenetration assumption is removed.
In the final part of the paper, we outline the results that can be obtained when dislocations are accounted for.
\par
Following the ideas of \cite{LPS}, in dimension $N=2$ we introduce other possible models
where the reference configuration represents a lattice with dislocations. 
More precisely, we fix $\rho\in[\lambda,1]$ and set
\[
 \LL_{\e}(\rho,k):=\LL_{\e}^-(1,k)\cup\LL_{\e}^+(\rho,k) \,,
\]
where
\begin{align*}
\LL_{\e}^-(1,k) &:= \phantom{\rho} \e \Z^2 \cap \overline\Om_{k\e} \cap \{x_1<0\} \,,
\\
\LL_{\e}^+(\rho,k) &:= \rho \e \Z^2 \cap \overline\Om_{k\e} \cap \{x_1\ge0\} \,,
\end{align*}
and $\overline\Om_\e$ is as in \eqref{nano}.
For $\rho\neq1$, the number of atomic layers parallel to $e_1$ is different in the two sublattices (for sufficiently large $k$); this can be regarded as a system containing dislocations at the interface.
\par
In presence of dislocations, the choice of the interactions and of the equilibria strongly depends on the lattice that one intends to model.
Therefore, in this section we focus on the simplest situation of hexagonal (or equilateral triangular) Bravais lattices in dimension two
and we fix
\[
H:=\begin{pmatrix}
1 & -\tfrac12 \\
0 & \tfrac{\sqrt3}2
\end{pmatrix} \,.
\]
The lattice $H \LL_{\e}(\rho,k)$
consists of two Bravais hexagonal sublattices with different lattice constants $\e$ and $\rho\e$, respectively;
see Figure \ref{fig:lattice} and Remark \ref{rmk:lattices}.
\par
\begin{figure} 
\centering
\psfrag{H}{\hspace{.1em}$H^{-1}$}
\psfrag{e}{\hspace{-.1em}$\eps$}
\psfrag{le}{\hspace{-.4em}$\rho\eps$}
\includegraphics[width=.95\textwidth]{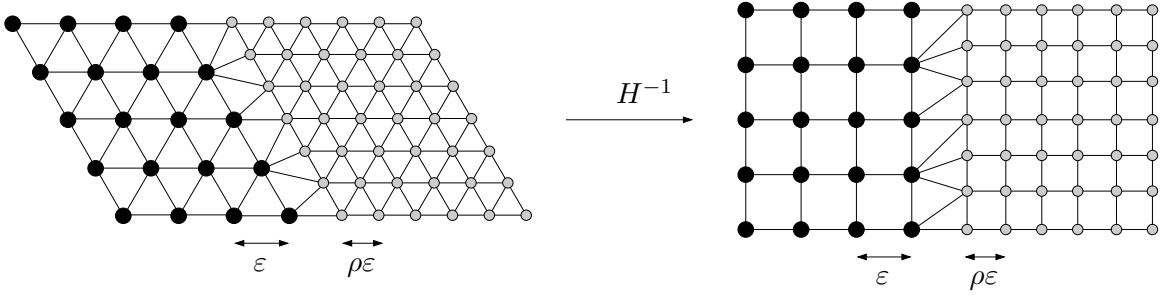}
\caption{Lattices with dislocations: choice of the interfacial nearest neighbours in $\LL_{\e}(\rho,k)$ and $H \LL_{\e}(\rho,k)$ for a Delaunay triangulation.}
\label{fig:lattice}
\end{figure}
The bonds between nearest and next-to-nearest neighbours are defined first in the lattice $H \LL_{\e}(\rho,k)$.
To this end, one chooses a Delaunay triangulation of $H \LL_{\e}(\rho,k)$ as defined in \cite[Section 1]{LPS}.
Two points $x,y$ of the lattice are said to be nearest neighbours if there is a lattice point $z$ such that the triangle $[x,y,z]$ is an element of the triangulation.
Two points $x,y$ are next-to-nearest neighbours if there are $z_1,z_2$ such that $[x,z_1,z_2]$ and $[y,z_1,z_2]$ are elements of the triangulation.
These definitions coincide with the usual notions of nearest and next-to-nearest neighbours away from the interface.
%
We underline that other choices of interfacial bonds are possible to derive our main results.
Indeed, one may start from any triangulation of the lattice satisfying the following properties:
the number of nearest neighbours of each point has to be uniformly bounded by a constant independent of $\e$,
while the length of the bonds in $H \LL_{\e}(\rho,k)$ has to be uniformly bounded by a constant $C_\e=C\e$.
\par
Once the bonds in the lattice $H \LL_{\e}(\rho,k)$  are defined,
we define the bonds of a point $x\in\LL_{\e}(\rho,k)$ as follows:
\begin{align*}
 B_1(x) &:= \{ \xi\in\R^N\colon Hx,\ H(x{+}\xi) \in  H \LL_{\e}(\rho,k) \ \text{are nearest neighbours} \} \,,
\\
 B_2(x) &:= \{ \xi\in\R^N\colon Hx,\ H(x{+}\xi) \in  H \LL_{\e}(\rho,k) \ \text{are next-to-nearest neighbours} \} \,.
\end{align*}
We remark that if $x_1\le-2\e$, then $B_1(x)=B_1$ and $B_2(x)=B_2$, while 
if $x_1\ge\rho\e$, then $B_1(x)=\rho B_1$ and $B_2(x)=\rho B_2$,  where $B_1,B_2$ are as in \eqref{bonds1}--\eqref{bonds2}. 
The total interaction energy is
\[
\begin{split}
\E\unolambda_{\e}(u,\rho,k) :=&\!\!\!
 \sum_{\substack{
x\in \L^-_{\e}(\rho,k)\\  
\xi\in B_1(x)
 }}
\!\!\!\
c_1
\left|\frac{|u(x+\xi)-u(x)|}{\e}-1\right|^p 
+ 
\!\!\!
\sum_{\substack{
x\in \L^+_{\e}(\rho,k)\\  
\xi\in B_1(x) 
 }}
 \!\!\!
c_1
\left|\frac{|u(x+ \xi)-u(x)|}{\e}-\lambda \right|^p
\\
&+\!\!\!
 \sum_{\substack{
x\in \L^-_{\e}(\rho,k)\\  
\xi\in B_2(x)
 }}
\!\!\!\
c_2
\left|\frac{|u(x+\xi)-u(x)|}{\e}-\sqrt3\right|^p 
+ 
\!\!\!
\sum_{\substack{
x\in \L^+_{\e}(\rho,k)\\  
\xi\in B_2(x) 
 }}
 \!\!\!
c_2
\left|\frac{|u(x+ \xi)-u(x)|}{\e}-\sqrt3\lambda \right|^p \,.
\end{split}
\]
Notice that away from the interface all bonds in the reference configuration are in equilibrium if $\rho = \lambda$;
instead, interfacial bonds are never  in equilibrium.
The equilibrium distance of two atoms at the interface is in fact an average of the equilibrium distances of the two sublattices.
Generalisations of this energy can be considered as described in \cite[Section 4]{LPS}.
\par
The results shown in detail in this paper for the defect-free case
(corresponding to $\rho=1$)
can be extended to models with dislocations ($\rho\neq1$)
without significant changes in the proof.
Thus we obtain a $\Gamma$-convergence result for the rescaled functionals
$\I\unolambda_{\e}(\cdot,\rho,k)$ defined as in \eqref{funzionale-risc}.
(Notice that the definition of the admissible functions is given as in the dislocation-free case, with the only variant that the gradients are constant on the elements of the triangulation introduced to define the interfacial bonds.)
Before stating the theorem we introduce the lattices
\begin{align*}
 \LL_{\infty}(\rho,k)&:=\LL_{\infty}^-(1,k)\cup\LL_{\infty}^+(\rho,k) \,, \\
\LL_{\infty}^-(1,k) &:= \phantom{\rho} \Z^2 \cap \overline\Om_{k,\infty} \cap \{x_1<0\} \,,
\\
\LL_{\infty}^+(\rho,k) &:= \rho \Z^2 \cap \overline\Om_{k,\infty} \cap \{x_1\ge0\} \,,
\end{align*}
where the triangulation is chosen in analogy with the one for $\LL_{\e}(\rho,k)$.
We also set
\[
\begin{split}
\gamma(P_1,\lambda P_2;\rho,k):=\inf\big\{ &\E\unolambda_{\infty}(v,\rho,k) \colon M>0\,,
\  v\in \A_{\infty}(\Om_{k,\infty})\,, \\
& \D v=P_1H \ \text{for} \ x_1\in(-\infty,-M)\,,\, \ \D v= \tfrac\lambda\rho P_2 H \ \text{for} \ x_1\in(M,+\infty) 
\big\}\,,
\end{split}
\]
with
\[
\begin{split}
\E\unolambda_{\infty}(u,\rho,k) :=&\!\!\!
 \sum_{\substack{
x\in \L^-_{\infty}(\rho,k)\\  
\xi\in B_1(x)
 }}
\!\!\!\
c_1
\Big||u(x+\xi)-u(x)|-1\Big|^p 
+ 
\!\!\!
\sum_{\substack{
x\in \L^+_{\infty}(\rho,k)\\  
\xi\in B_1(x) 
 }}
 \!\!\!
c_1
\Big||u(x+ \xi)-u(x)|-\lambda \Big|^p
\\
&+\!\!\!
 \sum_{\substack{
x\in \L^-_{\infty}(\rho,k)\\  
\xi\in B_2(x)
 }}
\!\!\!\
c_2
\Big||u(x+\xi)-u(x)|-\sqrt3\Big|^p 
+ 
\!\!\!
\sum_{\substack{
x\in \L^+_{\infty}(\rho,k)\\  
\xi\in B_2(x) 
 }}
 \!\!\!
c_2 
\Big||u(x+ \xi)-u(x)|-\sqrt3\lambda \Big|^p \,.
\end{split}
\]
\begin{theorem}\label{thm:disl}
The sequence of functionals $\{  \I\unolambda_{\e}(\cdot,\rho,k)\}$ 
$\Gamma$-converges, as $\e\to 0^{+}$, to the functional
\[
\I\unolambda(u,\rho,k)=
\begin{cases}
\ga(\rho,k)   &  \text{if}\ u\in\A\unolambda(\rho,k)\,,\\
+\infty   &  \text{otherwise,}
\end{cases}
\]
with respect to the weak* convergence in $W^{1,\infty}(\Om_k;\R^N)$, where
\[
\begin{split}
\A\unolambda(\rho,k):=\big\{
u\in W^{1,\infty}(\Om_k;\R^N)\colon &
\partial_{2} u  = 0 \ \text{a.e.\ in}\ \Om_k\,, \\
& |\partial_1 u |\leq 1 \ \text{a.e.\ in}\ \Om_k^- \,, \ 
|\partial_1 u |\leq \tfrac\lambda\rho \ \text{a.e.\ in}\ \Om_k^+
\big\} 
\end{split} 
\]
and
\[
\gamma(\rho,k):= \min\big\{ \gamma(I,\lambda I;\rho,k) , \gamma(I,\lambda J;\rho,k) \big\} \,.
\]
\end{theorem}
\par
%
The stronger topology introduced in Theorem \ref{thm4}  allows us to take into account the cost of ``folding'' the lattice using reflections, giving deeper insight into deformations that bridge different equilibria. Indeed, it is possible to combine Theorems \ref{thm4} and \ref{thm:disl} giving the $\Gamma$-convergence in the stronger topology for models with dislocations; we omit the full statement for brevity.
\par
\begin{remark}\label{rem:dislo}
It is easy to see that for $\rho=\lambda$
\[
 C_1 k\le \ga(\lambda,k)\le C_2 k
\]
for some constants $C_1,C_2>0$.
To obtain the estimate from above it is sufficient to consider the identical deformation
and recall that the maximal length of a bond and the maximal number of bonds per atom in the lattice $\LL_{\infty}(\rho,k)$ are uniformly bounded.
This configuration corresponds to the case when dislocations are uniformly distributed along the interface between the two sublattices. 
(Recall that here $N=2$ and that the length of the interface is $2k$.)
In contrast, the cost of a defect-free configuration ($\rho=1$) is superlinear as already shown in Theorem \ref{thm:scaling}.
In fact, following the same proof it is possible to conclude that whenever $\rho\neq\lambda$ one has
\[
\ga(\rho,k)\le C_\rho \, k^2 \quad \text{and} \quad \lim_{k\to\infty}\frac{\ga(\rho,k)}{k}=+\infty \,.
\]
This gives a mathematical proof of the experimentally observed fact that dislocations are preferred
in order to relieve the lattice mismatch when the thickness of the
specimen is sufficiently large. We recall that a similar result was proven in \cite{LPS,LPS2} (under the non-interpenetration assumption), see also Remark \ref{rmk:nonint}.
\end{remark}
The results sketched here for hexagonal lattices can be obtained also for other lattices by adapting the technique to each specific case.
In particular, we refer to \cite{LPS2} for details on the rigidity of face-centred and body-centred cubic lattices in dimension three. 
%
%
%
%
%
%
%
%
%
%
\section*{Acknowledgements} \noindent
The authors thank Sergio Conti for useful discussions.
This work has been partially supported by the European Research Council through the Advanced Grant No.\ 290888.
\end{document}